\documentclass[onefignum,onetabnum]{siamart250211}

\usepackage{amsmath, amsfonts,amssymb,color}
\usepackage[english]{babel}

\usepackage{dsfont}
\usepackage{mathabx}

\usepackage{enumitem}
\usepackage{graphicx,epsfig}
\usepackage{float}
\usepackage{capt-of}
\usepackage{tikz}

\newcommand{\R}{{\mathbb R}}
\newcommand{\N}{{\mathbb N}}

\newcommand{\one}{{\mathds{1}}}

\newcommand{\cP}{{\mathcal P}}

\newcommand{\cM}{{\mathcal M}}
\newcommand{\cN}{{\mathcal N}}

\newcommand{\ds}{\displaystyle}

 \definecolor{mypurple}{RGB}{140,0,255}
\definecolor{myred}{rgb}{255,0,0}
\definecolor{mydarkturquoise}{RGB}{0,206,209}
\definecolor{mydeeppink}{RGB}{255,20,147}
\definecolor{darkblue}{RGB}{0,0,140}
\definecolor{blue2}{RGB}{0,0,0}
\definecolor{middleblue}{RGB}{0,0,71}
\definecolor{light-gray}{gray}{0.9}
\definecolor{ProcessBlue}{cmyk}{1,0,0,0.40}
\definecolor{Black}{cmyk}{0,0,0,1}
\definecolor{Red}{cmyk}{0,1,1,0.2}
\definecolor{Green}{cmyk}{0.9,0,1,0}
\definecolor{Orange}{cmyk}{0,0.61,0.87,0.5}
\definecolor{Fuchsia}{cmyk}{0.47,0.91,0,0.06}
\definecolor{PineGreen}{cmyk}{0.92,0,0.59,0.30}

\numberwithin{equation}{section}
 \def\theequation{\thesection.\arabic{equation}}
 \newtheorem{remark}{Remark}[section]
 \newtheorem{assumption}{Assumption}[section]

\begin{document}

\title{State constrained optimal control problems with control on the acceleration.\\ Applications to kinetic mean field games}
\author{
  Yves Achdou \thanks { Universit{\'e}  Paris Cit{\'e} and  Sorbonne Universit{\'e}, CNRS, Laboratoire Jacques-Louis Lions, (LJLL), F-75006 Paris, France, achdou@ljll-univ-paris-diderot.fr}}

\maketitle

\begin{abstract}
   Relying on the careful study of a related problem in the calculus of variations,
   we study  a class of optimal control problems in which the control lies on the acceleration, with state constraints on the position variable. In dimension one, we find  explicit formulas in the special case when the running cost is a power of the acceleration (in absolute value) and the terminal cost is zero.
   For more general costs or/and  higher dimensions, we study
   the singularities of the value function. We also prove the closedness (in the $C^1$ topology) of the graph of the multivalued mapping which maps a point in the state space  to the set of optimal trajectories which start from  this point. A consequence of the latter is the existence,  under general assumptions,  of relaxed equilibria for a class of kinetic mean field games with state constraints.
 \end{abstract}
 


\section{Introduction}
\label{sec:intro}

The main purpose of  the present paper, which is the continuation of the  article \cite{MR4444572},
is
the study of deterministic  optimal control problems  in which  the agents control their acceleration and must stay in  $\overline \Omega$, where $\Omega$ is an open  subset of $\R^n$ with a sufficiently smooth boundary. The state variable is the pair $(x,v)$ of position and velocity, 
and the state space is of the form $\Xi= \overline \Omega\times \R^n$. The control variable, i.e. the acceleration takes its values in the full space $\R ^n$.

We  assume that the cost of a trajectory $\xi$  is of the form
\begin{displaymath}
  J_t(\xi)=
  \frac 1 q \int_t^T \left|\frac {d^2\xi} { dt^2}(s)\right|^q ds + \int_t^T \ell \left(\xi(s), \frac {d\xi} { dt}(s),s\right) ds + g\left(\xi(T), \frac {d\xi} { dt}(T)\right),
\end{displaymath}
where $\ell$ and $g$ are bounded and continuous functions and $q$ is a real number greater than $1$. The value function is $u(t,x,v)=\inf J_t(\xi)$, where the infimum is taken on the trajectories which belong to $W^{2,q}(t,T)$ and such that $\xi(t)=x$, $ \frac {d\xi} { dt}(t)=v$ and  $\xi(s)\in \overline \Omega$ for all times $s\in [t,T]$.

The value function  $u$  blows up at every $(x,v)$ such that $x\in \partial \Omega$ and $v$ points outward $\Omega$, and the function $(x,v) \mapsto u(t, x,v)$ is discontinuous at every $(x,v)$ such that  $x$ belongs to $ \partial \Omega$ and $v$ is tangent to $\partial \Omega$ at $x$. This phenomenon can be understood as follows: because the agents may have to brake hard to avoid exiting the domain,  the value function
blows up at those points $(x,v)$ where $x\in \partial \Omega$ and $v$ points outward $\Omega$. A large part of the present paper is devoted to getting accurate information on the singularities of $u$ at the boundary of $\Omega$. Although our setting is deterministic, this objective may be reminiscent of \cite{MR990591}, in which the authors studied the singularities arising in state constrained  optimal stochastic control problems.

Here, it is not easy to rely on PDE based methods for studying the value function. More precisely,  the controllability conditions at the boundary proposed  in  the seminal articles \cite{MR838056, MR951880} do not hold, and  it is not possible to characterize the value function as the continuous viscosity solution of a  Hamilton-Jacobi equation in the interior of the state space that furthermore satisfies  a supersolution condition at the boundary. One cannot either rely on the results contained in   e.g. \cite{MR1200233, MR1794772} on lower semi-continuous viscosity solutions of Hamilton-Jacobi equations with state constraints, because 
the  controllability assumptions made there  (that are different from those in  \cite{MR838056, MR951880}) are not satisfied and the running cost is not bounded.
Without any controllability hypothesis, the value function is discontinuous and its characterization with PDEs becomes more difficult, see \cite{MR1672758,MR2859862}. In the latter two articles, assumptions are made so that the value function is bounded,
and  to the best of our knowledge, these articles have not been generalized to include the
problem studied presently. None of the latter references contains the information on the singularities that we are interested in. This is why we will mainly use techniques from the calculus of variations and explicit solutions when they exist.

As already said, this paper is a continuation of  \cite{MR4444572},  which was mainly devoted to  deterministic mean field games (MFGs in short) with control on the acceleration and state constraints, or in other words, {\sl kinetic} MFGs with state constraints (we will soon give comments on this terminology). In \cite{MR4444572}, an important step was to study the related optimal control problems. The present paper contains  improvements of the results contained in \cite{MR4444572}, both on optimal control and on MFGs. 

More precisely,   \cite{MR4444572} was devoted to relaxed equilibria that are described by a probability measure on trajectories. Such a notion proves useful when it is difficult to characterize the equilibria with systems of PDEs as it was done in the  pioneering works
\cite{MR2269875,MR2271747,MR2295621} of  Lasry and Lions. The difficulties in using PDEs
often arise from boundary conditions and concentrations of the density of states near  some points at the boundary.
Relaxed equilibria for MFGs  with state constraints were studied in \cite{MR3888967} under strong controllability assumptions, following ideas contained in \cite{Br89, Br93,MR3556062}. 

On the other hand, the same  deterministic  MFGs, this time   
in the absence of state constraints, and the related systems of PDEs, were investigated in  \cite{MR4102464}, see also \cite{MR4132067}. Assuming that the control of the agents lies on the acceleration 
is  natural in the context of traffic flows. Second order dynamics are also relevant in economics:  for example, the percentage of time spent by an agent in education can be seen as a control variable  which acts on the  second time derivative of her wealth, see \cite{LUCAS19883}.
In \cite{MR4835150}, the authors investigated numerical methods that can in particular be applied  to the problem studied in \cite{MR4102464}.
  MFGs similar to the ones considered here but with a special coupling cost and without state constraints  were addressed in \cite{MR4199409} in relation with flocking models, and it would be  interesting  to generalize the results  in \cite{MR4199409} by incorporating state constraints.  Finally, the same kind of MFGs, this time  with  local  coupling costs or with costs that do not depend separately on the control and on the  distribution of states, were addressed in \cite{MR4492922,MR4920379}. Besides, the denomination {\sl kinetic mean field games} was introduced in the latter references, by analogy with kinetic models in fluid mechanics in which the velocity is also a state variable.

  Different and very reasonable models suppose  that the acceleration takes its values in a compact subset or $\R^n$. We have not made this assumption here, because we have in mind situations in which accurate information on the space of controls is not available (this may happen in economics for example).
 We hope that the cases in which we provide closed formula or accurate estimates will become useful model problems (for further theory or for numerical methods).

The present paper contains new results and also improvements of the results contained in \cite{MR4444572}:
\begin{description}[style=unboxed,leftmargin=0cm]
\item[1] For $n=1$ and  $\overline \Omega=[-1,1]$:
\begin{description}
\item[(a)] When  $\ell=0$ and $g=0$, we obtain an explicit formula for the value function, by using arguments from the calculus of variations, see Theorem \ref{val_func_char} and Corollary \ref{val_func_char_2} below. We were not able to find this formula in the  literature.
\item[(b)] For  general $\ell$ and $g$, we study the behaviour of the value function near the points $(x,v)$, $x=\pm 1$ and $v=0$, at which the value function is singular
\item[(c)] We improve the results on the closedness (with respect to the $C^1$ topology) of the graph of the multivalued map which maps a point $(x,v)$ of the state space to the set of the  optimal trajectories with initial conditions $(x,v)$. The improvements lie in two aspects: a) we get rid of an assumption that was made in  \cite[Lemma 4.6]{MR4444572} on $\ell$, namely that $\ell$ did not reward the trajectories that exit the domain. b) the new results hold for any $q>1$ while $q=2$ was assumed  in \cite[Lemma 4.6]{MR4444572}.
\item[(d)] The new closed graph property allows us to improve the theorems on the existence of relaxed equilibria for MFGs  that are contained  in \cite{MR4444572}. The present results hold under minimal hypothesis.
\end{description}

\item[2.] For $n>1$,  $q=2$ and $\Omega $ a convex and bounded open set with a $C^3$ boundary:
  \begin{description}
  \item[(a)] We obtain a general closed graph property similar to the one mentioned in point 1.(c) above, still relying on the explicit formulas obtained in dimension one. In  \cite{MR4444572}, the closed graph property was obtained only for a class of compact subsets of $\overline \Omega\times \R^n$ satisfying a special assumption
  \item[(b)] As a consequence of  this latter result, we obtain the  existence of relaxed equilibria for MFGs.
  \end{description}
  In dimension $n>1$, we restrict ourselves to  $q=2$ because some intermediate results rely on a sharp estimate that is still lacking for $q\not=2$.
\end{description}
All our results rely on  the careful study of an auxiliary problem in the calculus of variations. The latter  had already been addressed in  \cite{MR4444572} but more rapidly and only for $q=2$, and it was not the cornerstone of  \cite{MR4444572} as it is of  the present paper.

The paper is organized as follows.
In the remainder of Section \ref{sec:intro}, we describe the class of optimal control problems and review some results that where obtained in \cite{MR4444572} and will be useful in the present paper. 
Section \ref{sec:auxil_opt_cont_pb} is devoted to finding the explicit solution of the optimal control problem  mentioned above when $n=1$, $\Omega=(-\infty, 0)$  or $\Omega=(-1,1)$, and both $\ell$ and $g$ are identically $0$. The construction relies on the study  in Subsection \ref{sec:auxil_pb} of an auxiliary problem in the calculus of variations. The results found in Subsection \ref{sec:auxil_pb} will be key in all the remainder of the paper.
 Since their proofs are rather technical and may be skipped on first reading, they will be written in appendix.
The explicit formulas are given in Subsection \ref{sec:auxil_opt_cont_pb_2}. The same optimal control problem  but with general $g$ and $\ell$ is tackled in Section \ref{sec:sc_pb}, and in particular, the  closed graph properties mentioned above are proved in Subsection \ref{closed_graph_1}. 
Section \ref{sec:multid} is devoted to the quadratic case ($q=2$) when $\Omega$ is a bounded and convex domain in $\R^n$ with a $C^3$ boundary. There, the main result concerns 
the closed graph properties mentioned above.
Finally, in Section \ref{sec:mfg}, relying on the above-mentioned properties and the arguments contained in \cite{MR4444572}, we obtain the existence of relaxed equilibria for related MFGs with state constraints, under general hypothesis.

\subsection{Description of the optimal control problem and known facts}
\label{sec:setting}
$\;$
\paragraph{\bf Notation}

 Given two real numbers $a$ and $\gamma>0$, we define $ a_+^\gamma$ by
  \begin{displaymath}
    a_+^\gamma=(\max(a,0))^\gamma.
  \end{displaymath}

\paragraph{\bf Setting and assumptions}
Let  $\Omega$ be a  bounded and convex domain of $\R^n$  with a boundary $\partial \Omega$ of class $C^3$. For $x\in \partial \Omega$, let $n(x)$ be the unitary vector normal to $\partial \Omega$  pointing outward $\Omega$. We will use the signed distance to  $\partial \Omega$, $d: \R^n\to \R$,
\begin{displaymath}
 d(x)=\left\{
    \begin{array}[c]{rcl}
      \min_{y\in \partial \Omega} |x-y|,\quad&\hbox{if}&\quad x\notin \Omega,\\
      -\min_{y\in \partial \Omega} |x-y|,\quad&\hbox{if}&\quad x\in \Omega.
    \end{array}
\right.
\end{displaymath}
Since $\partial \Omega$ is $C^3$, the function $d$ is $C^3$ in a neighborhood of  $\partial \Omega$, see \cite{MR1814364}, and for all $x\in \partial \Omega$, $\nabla d(x)= n(x)$.

The state space is  $\Xi=\overline \Omega \times \R^n$.
Given a time horizon $T>0$,
the set of admissible trajectories after time $t$, $t\in [0,T)$, is
\begin{equation}
\label{eq:1}
\Gamma_t=\left\{
    \xi\in C^1([t,T];\R^n) : \left|
\begin{array}[c]{l}   \ds \xi'\in AC([t,T];\R^n) \\  \ds  \xi(s)\in \overline \Omega,  \forall s\in [t,T] 
\end{array}\right.
\right\}.
\end{equation}
Let us set $\Gamma=\Gamma_0$;   $\Gamma$ is a metric space with the distance $d(  \xi, \tilde\xi)= \|\xi-\tilde\xi\|_{  C^1([0,T];\R^n)}$.
For any $(x,v)\in \Xi$, and $t\in [0,T]$, we also define
\begin{equation}
\label{eq:3}
\Gamma_t[x,v]=\{ (\xi, \eta) \in \Gamma_t:\,\xi(t)=x,\, \xi'(t)=v\},
\end{equation}
and set $\Gamma[x,v]=\Gamma_0[x,v]$.
The set  $\Gamma_t[x,v]$ is non-empty if and only if $(x,v)\in \Xi^{\rm{ad}}$ where
  \begin{equation}
    \label{eq:5}
\Xi^{\rm{ad}}=\{  (x,v):  x\in \overline \Omega, \;  v\cdot n(x)\le 0 \hbox{ if } x\in \partial \Omega
 \}\subset \Xi.
\end{equation}
The optimal control problem studied in the present paper involves a running cost $\ell: \Xi\times [0,T]\to \R$ and  a terminal cost $g: \Xi\to \R$.
\begin{assumption}
  \label{sec:optim-contr-probl-2}
  The running cost $\ell: \Xi\times [0,T]\to \R$  and the terminal cost $g: \Xi\to \R$ are 
  bounded and continuous functions.  Set $M= \|g\|_{L^\infty ( \Xi) }  +   \|\ell\|_{L^\infty( \Xi \times [0,T]) } $.
\end{assumption}
Given $q>1$,  the cost associated to an admissible trajectory $\xi\in \Gamma_t$ is
\begin{equation}
  \label{eq:costJ}
  J_t(\xi)=\left\{
    \begin{array}[c]{rl}
   \ds   \int_t^T \left(\ell(\xi(s),\xi'(s) ,s) +\frac 1 q   |\xi''(s)|^q
      \right) ds+g(\xi(T), \xi'(T))   &\hbox{ if } \xi\in W^{2,q}(t,T), \\
      +\infty    &\hbox{ otherwise. }
    \end{array}
\right.
\end{equation}
We set $J=J_0$. For $t\in [0,T]$ and $(x,v)\in  \Xi^{\rm{ad}}$, the value of the control problem is
\begin{equation}
\label{eq:6}
u(t,x,v) = \inf_{\xi \in \Gamma_t[x,v]} J_t(\xi).
\end{equation}
Let $\Gamma_t^{\rm{opt}}[x,v]$ be the set of all
$ \xi\in\Gamma_t[x,v]\cap  W^{2,q} (0,T)$ which achieve  $u(t,x,v)=J_t(\xi)$.
  For  $(x,v)\in  \Xi^{\rm{ad}}$, one can prove that  $\Gamma_t^{\rm{opt}}[x,v]$ is non empty   by observing first that the optimal value is finite, see \cite[Lemma 2.1]{{MR4444572}}, then by using the direct method in the  calculus of variations. It is also standard to show  that  $\Gamma_t^{\rm{opt}}[x,v]$ is closed.
For brevity, let us set  $\Gamma^{\rm{opt}}[x, v]= \Gamma_0^{\rm{opt}}[x, v]$.

 With  $p= q/(q-1)$, it is known that the value function $u$ of \eqref{eq:6} solves
\begin{eqnarray}
  \label{eq:1000}
  -\partial_t u-v\cdot D_xu+ \frac 1 p | D_v u|^p= \ell(x,v,t)
  \textrm{ in } (0,T)\times \Omega \times \R^n,\\
   \label{eq:1001}
  u(T, x,v)=g(x,v) 
  ,\textrm{ in }\Omega \times \R^n,
      \end{eqnarray}
  (\eqref{eq:1000} is understood in the sense of viscosity);
  \eqref{eq:1000} and \eqref{eq:1001}  should be  supplemented  with boundary conditions linked to state constraints.
  
 \begin{remark}
    The results in Sections  \ref{sec:sc_pb} and \ref{sec:multid} below hold for costs of the type
    \begin{displaymath}
    \int_t^T \left(\ell(\xi(s),\xi'(s) ,s) + a(\xi(s),\xi'(s), s)  |\xi''(s)|^q
 \right) ds+g(\xi(T), \xi'(T)),      
    \end{displaymath}
where $a$ is continuous and bounded, bounded from below by a constant $\underline{a}>0$.
  \end{remark}

When $n=1$, we will restrict ourselves to $\Omega=(-1,1)$. Then,
  \begin{equation}
    \label{eq:7}
 \Xi=[-1,1]\times \R,\quad \hbox{ and }\quad   \Xi^{\rm{ad}}= [-1,1)\times [0,+\infty) \cup  (-1,1]\times (-\infty,0].
\end{equation}

\paragraph{\bf Known facts}
\label{sec:bounds3}

We place ourselves in the framework defined above.

Lemma \ref{sec:bounds_and_continuity} and Proposition \ref{sec:bounds_opt_traj_nd} that follow have  been proved in \cite{MR4444572}
(only for $q=2$ but the generalization to $q>1$ is easy),
see  \cite[Lemma 2.3]{MR4444572},     \cite[Prop. 2.7]{MR4444572} for $n=1$ and 
\cite[Prop 4.7]{MR4444572} for $n>1$ and their proofs. Lemma  \ref{sec:bounds_and_continuity}  (respectively  Proposition  \ref{sec:bounds_opt_traj_nd})
will be useful to prove Theorem \ref{prop:clos-graph-prop}
(respectively Theorem  \ref{sec:mean-field-games-5})
below.
\begin{lemma}\label{sec:bounds_and_continuity}
  Consider a compact subset $\Theta$ of $\Xi^{\rm ad}$ (see \eqref{eq:5})  which has the following property:  for all sequences  $(x^i, v^i)_{i\in \N}$  with values  $\Theta$ such that  $\ds \lim_{i\to +\infty} (x^i, v^i)=(x, v)\in \Theta$, the following holds: if  $x\in  \partial \Omega$ and $v\cdot n(x)=0$,  then
\begin{equation}
  \label{eq:8}
(v^i \cdot \nabla d (x^i) )_+^{2q-1} = o\left( \left |d(x^i)\right|^{q-1} \right),
\end{equation}
with the convention that if $x^i\in \partial \Omega$ and  $v^i \cdot \nabla d (x^i) =v^i\cdot n(x^i)   \le 0$, then the quotient $(v^i \cdot \nabla d (x^i) )_+^{2q-1} / \left |d(x^i)\right|^{q-1}$ takes the value $0$. Note that $ \nabla d(x^i)$ is meaningful for $i$ large enough, because $d$ is $C^3$ near $\partial \Omega$.

Under Assumption ~\ref{sec:optim-contr-probl-2}, for all $\tau\in (0,T)$, the function $[0,T-\tau]\times \Theta \ni (t,x,v) \mapsto  u(t, x,v)$ ($u$ is defined in \eqref{eq:6}) is continuous and bounded by a constant depending on $\tau$.
\end{lemma}

 \begin{proposition}\label{sec:bounds_opt_traj_nd}
   For positive numbers $r$ and $C$,  let us set
   \begin{eqnarray}
  \label{eq:9}   \Theta_r&=&\left\{(x,v)\in  \Xi:       (v\cdot \nabla d(x) )_+ ^{2q-1} \le r |d(x)|^{q-1} \right\},\\
      \label{eq:10} K_C&=& \{  (x,v)\in  \Xi: \; |v|\le C\},   
     \\
   \label{eq:11}   \Gamma_C &=& \left\{  (\xi,\eta)\in \Gamma: \left|
         \begin{array}[c]{l}
           (\xi(t), \xi'(t))\in K_C, \quad \forall t\in [0,T],\\
        \left \| \frac {d^2\xi}{dt^2} \right\|_{L^q(0,T;\R^n)}\le C
      \end{array}
    \right.  \right\}.
   \end{eqnarray}
   Under Assumption~\ref{sec:optim-contr-probl-2},
for all $r>0$, there exists a positive number $C= C(r, M)$  such that
if  $(x,v)\in \Theta_r$, then $\Gamma^{\rm opt} [x,v]\subset  \Gamma_C$.
Moreover, as $r\to +\infty$,  $C(r,M)= O(r ^{1/q})$.
\end{proposition}
\begin{remark} \label{sec:bounds_opt_traj_1d}
  In the case when $n=1$ and $\Omega=(-1,1)$,  \eqref{eq:9}
 reads
  \begin{equation}
     \label{eq:12}
 \Theta_r=\left\{(x,v)\in  \Xi:     -r (x+1)^{q-1}  \le |v|^{2q- 2} v\le r (1-x)^{q -1} \right\}. \end{equation}
\end{remark}

\section{ Explicit formulas for the value function in \eqref{eq:6} in the case $n=1$, $\ell=0$ and $g=0$}
\label{sec:auxil_opt_cont_pb}
\subsection{An auxiliary  variational problem}
\label{sec:auxil_pb}

For $q>1$, let us introduce an  auxiliary function $f: \; [1,+\infty)\to \R$, that will be used only in Section \ref{sec:auxil_pb}: 
      \begin{equation}
        \label{eq:13}
        f(y)= \left \{
          \begin{array}[c]{rcl}
           \ds -q  \frac {y^{\frac q {q-1}}} {y^{\frac q {q-1}}-(y-1)^{\frac q {q-1}}}
+(q-1) y, \quad  &\hbox{if }&y>1,\\
            1 , \quad  &\hbox{if }&y=1.
          \end{array}
\right.
      \end{equation}

\paragraph{\bf Formulation of the auxiliary variational problem and characterization of its optimal value}

Given a positive horizon $T>0$, the admissible trajectories
in our auxiliary variational problem
are the functions  $\xi\in W^{2,q}(0,T)$ such that $\xi (s) \le 0$ for all $s\in [0,T]$. For $\xi\in W^{2,q}(0,T)$, we will often set  $\eta(s)=\xi'(s)$ and $\alpha(s)=\xi''(s)$,
$\eta(s)$ and $\alpha(s)\in \R$ are respectively the velocity and  the acceleration at time $s$.

Consider  two velocities $w,v$, $0\le w< v$ and $\theta \in (0,T]$. In Proposition  \ref{sec:prop_1} below, we estimate the minimal cost
 $ \frac 1 q \int_0^\theta  \left|\frac {d^2 \xi}{dt^2} (s) \right| ^q ds$
 associated to admissible  trajectories  $\xi$ such that  $\xi(0)=x,\;\xi'(0)=v$ and  $\xi'(\theta)=w$.

\begin{proposition} \label{sec:prop_1}
  Consider $q\in (1,+\infty)$, $T>0$,  $x<0$, $v> 0$,  such that  $
 \frac {2q-1}{q-1} \frac {|x|} v< T$.
  Given  two real numbers  $\theta\in \left( 0 ,T\right]$ and  $w\in [0, v) $, 
 set
   \begin{equation}
   \label{eq:14}
     K_{\theta,w}=\left \{ \eta\in W^{1,q}(0,\theta;\R): \left|
       \begin{array}[c]{l}
         \eta(0)=v,\quad \eta(\theta)= w,\\
         \eta(s)\ge    w,\;   \forall s\in [0,\theta],\\
       \ds   x+\int_0^\theta \eta(s) ds \le 0
       \end{array}\right. \right\}.
 \end{equation}
 (i) The optimization problem
 \begin{equation}
     \label{eq:15}
     I(\theta,w)= \inf_{\eta\in K_{\theta,w} }  \frac 1 q \int_0^\theta  \left|\frac {d\eta}{dt} (s) \right| ^q ds 
   \end{equation}
   has a unique minimizer denoted $\eta_{\theta, w}$, and 
   \begin{equation}\label{eq:16}
     I(\theta,w)\!=\!\! \left | \!\!
       \begin{array}[c]{rl}
         \ds \frac 1 q  \frac {(v-w)^{q}} { \theta^{q-1}} ,
         \quad \quad\quad &\hbox{if }  \theta\in \left(0, \frac {2|x|}{v+w}\right], 
         \\ \\
         \ds \frac { \theta^{-(q-1)}} {2q-1} \left( \frac {q (v-w)}{q-1}\right)^q
         \frac { X^{\frac {2q-1}{q-1}} -(X-1)^{\frac {2q-1}{q-1}} }{ \left( X^{\frac {q}{q-1}} -(X-1)^{\frac {q}{q-1}}\right)^q}, &  \hbox{if } \theta\in  \left( \frac {2|x|}{v+w},   \frac {(2q-1)|x|}{(q-1)v+qw} \right],
\\ \\
         \ds \frac {q^{q-1}} {(2q-1)^q}   \frac {(v-w)^{2q-1}} { (|x|-\theta w)^{q-1}}
 , \quad\quad \quad &\hbox{if } \theta \in  \left(  \frac {(2q-1)|x|}{(q-1)v+qw}  ,T\right],
  \end{array}\right.
\end{equation}
where,  with $f$ defined in \eqref{eq:13},  $X$ is the unique solution in $[1,+\infty)$ of
  \begin{equation}\label{eq:17}
      f(X)=  -\frac 1{v-w} \left((2q-1)\frac x \theta +(q-1)w +qv\right).
  \end{equation}  
\noindent
  (ii) The quantity $I(\theta,w)$   depends continuously on $ (\theta, w) \in (0,T]\times [0,v)$.

  \medskip
  
 \noindent (iii) The minimizer $\eta_{\theta, w}$ has an explicit form  if $ \theta\in \left  (0,  \frac {2|x|}{v+w}\right]\cup \left[  \frac {(2q-1)|x|}{(q-1)v+qw}  ,T\right]$.

 \medskip

 \noindent
(iv) If $q=2$ and $ \theta \in  \left( \frac {2|x|}{v+w},   \frac {(2q-1)|x|}{(q-1)v+qw} \right)=\left( \frac {2|x|}{v+w},  \frac {3|x|}{v+2w}  \right)$, then
\eqref{eq:17}  has an explicit solution, and  the second line of \eqref{eq:16} becomes
\begin{equation}\label{eq:18}
 I(\theta,w)= 
  \ds 6 \frac {x^2} {\theta^3} + 6 \frac {x(v+w)}{\theta^2} +2 \frac { v^2+vw+w^2}\theta, \quad \hbox{ if }  \theta \in \left( \frac {2|x|}{v+w},  \frac {3|x|}{v+2w}  \right).
\end{equation}
 \end{proposition}

 \begin{proof}
   The proof, which consists of studying a related variational inequality with a $q$-laplacian, is written in appendix.
 \end{proof}

 \begin{remark} \label{sec:clos-graph-prop-6}
The partition of the interval $(0,T]$  in (\ref{eq:16}) is justified by the assumptions of Proposition~\ref{sec:prop_1}. Indeed
\begin{itemize}
\item  $   \frac {2q-1}{q-1} \frac {|x|} v< T$ and $w\ge 0$ imply that $ \frac {(2q-1)|x|}{(q-1)v+qw} < T$
\item
$  \frac {2|x|}{v+w} <  \frac {(2q-1)|x|}{(q-1)v+qw} $ 
 because $0\le w< v$.
\end{itemize}
Note also that  if $|x|/v\to 0$, then  $ \frac {(2q-1)|x|}{(q-1)v+qw}  \ll T$, because 
$\frac {(2q-1)|x|}{(q-1)v+qw} \le \frac {(2q-1) |x|} {(q-1)v}$.
\end{remark}

\paragraph{\bf Variations of $I(\theta, w)$ with respect to $w$ for a fixed $\theta \in   \left (\frac {|x|} v, 
 \frac {(2q-1) |x|}{ (q-1) v}
    \right ]$}

The following lemma will be important in order to obtain the desired explicit formula for the value function of \eqref{eq:6} in the case $n=1$, $\ell=0$ and $g=0$.
\begin{lemma}
  \label{lem_var_w}
  Take $q,T,x,v$ as in Proposition \ref{sec:prop_1}, and fix 
 $\theta \in   \left (\frac {|x|} v,   \frac {2q-1} {q-1} \frac {|x|} v  \right ]$ .

  The map $w\mapsto I(\theta, w)$ is decreasing in the interval  $\left[0,  \frac { (2q-1)|x|/\theta-(q-1)v } q \right]$ and increasing in the interval $ \left[   \frac { (2q-1)|x|/\theta-(q-1)v } q , v\right)$. Its minimal value is $\frac { (2q-1)^{q-1}}{q (q-1)^{q-1}} \frac { (v -|x|/\theta)^q }{\theta^{q-1}}$.
\end{lemma}
\begin{proof}
  The proof is written in appendix.
\end{proof}

\paragraph{\bf Study of $\inf_{\theta \in (0,T] }  I (\theta,w)$ for a fixed $w\in [0,v)$}

The following lemma  contains information on the function $\theta\mapsto I(\theta,0)$.
\begin{lemma}\label{lemma:ass_infimum0}
   Take $q,T,x,v$ as in Proposition \ref{sec:prop_1}.
  The function $\theta \mapsto I(\theta,0)$ is non-increasing on $[0,T]$.
  The minimal value of $\theta \mapsto I(\theta,0)$ on $(0,T)$ is  $\frac {q^{q-1}} {(2q-1)^q}   \frac { v^{2q-1}} { |x|
    ^{q-1}}$.
\end{lemma}
\begin{proof}
  We know from \eqref{eq:16} that $\theta\mapsto  I(\theta,0)$ is decreasing on $(0, 2|x|/v]$ and constant on $[\frac {2q-1} {q-1} \frac {|x|}{v},T]$.
  Consider $\theta,\theta'\in [ \frac{2|x|} v, \frac {2q-1} {q-1} \frac {|x|}{v}]$ such that $\theta'<\theta$.
  The function $\eta_{\theta ', 0}$  extended by $0$ in $[\theta', \theta]$ becomes an admissible competitor for the variational problem defining $ I(\theta, 0)$. Therefore $I(\theta,0)\le I(\theta', 0)$.
  
  Hence,  $\theta \mapsto I(\theta,0)$ is non-increasing on $[0,T]$ and the lemma is proved.
\end{proof}
The following proposition deals with the behaviour of  $\inf_{\theta\in (0,T]} I(\theta,w)$ as
$\frac {|x|}{v}\to 0$,  $\frac w v \to 0$ and $ \frac { v ^{2q-1}} { |x|    ^{q-1}}$ remains bounded from below by a positive constant. It will be useful to study the singularities of the value function in \eqref{eq:6}.
\begin{proposition}\label{lemma:ass_infimum}
  Consider  $q\in (1,+\infty)$, $T>0$. Let $(x_i, v_i)_{i\in \N}$ be such that  $x_i<0$, $v_i>0$,  $\frac {|x_i|}{v_i}\to 0$  and $ \frac { v_i ^{2q-1}} { |x_i|    ^{q-1}}>c$ for a given positive constant.
Let  $(w_i)_{i\in \N}$ be such that $0\le w_i\le v_i$ and $\frac {w_i}{v_i}\to 0$, and
let $I_i(\theta, w_i)$ be the quantity given by  \eqref{eq:16} for $x=x_i$, $v=v_i$, $w=w_i$. Then, 

(i) as $i\to \infty$,
\begin{equation}
  \label{eq:19}
  \inf_{\theta\in (0,T]} I_i(\theta,w_i)\sim \frac {q^{q-1}} {(2q-1)^q}   \frac {v_i^{2q-1}} { |x_i|    ^{q-1}}.
\end{equation}

(ii) If furthermore $q=2$, then as $\frac {v_i^2 w_i} {|x_i|}\to 0$,
\begin{equation}
    \label{eq:20}
    \inf _{\theta\in (0,T]}   I_i(\theta,w_i)
  = \frac 2 9 \frac {v_i^3} {|x_i|} +O\left(\frac {v_i^2 w_i} {|x_i|}\right).
\end{equation}

\end{proposition}
\begin{proof}
  The proof is written in  appendix.
\end{proof}

\begin{remark}
  It is easy to check  that the function $\psi(x,v)= \frac {q^{q-1}} {(2q-1)^q}   \frac { v ^{2q-1}} { |x|    ^{q-1}}$ solves
  \begin{displaymath}
    -v \frac {d\psi} {dx} (x,v) + \frac 1 p \left|  \frac {d\psi} {dv} (x,v) \right|^{p} =0
  \end{displaymath}
  in $\R_-\times \R_+$,  where $p= q/(q-1)$ is the conjugate exponent of $q$.
\end{remark}

We believe that the generalization of  point (ii) in Proposition \ref{lemma:ass_infimum}
to $q\not=2$ is true, i.e.  for any sequence  $(w_i)_{i\in \N}$ such that  $0\le w_i\le v_i$ and $\lim_{i\to \infty} \frac { w_i} {v_i}=0$, 
\begin{equation*}
    \inf   \left \{I_i(\theta,w_i),\; \theta \in  \left(0 ,T\right)\right\}
    =
 \ds \frac {q^{q-1}} {(2q-1)^q}   \frac {v_i^{2q-1}} { |x_i|^{q-1}}
 +O\left(    \frac {v_i^{2q-2} w_i} {|x_i|^{q-1}} \right),
\end{equation*}
but the proof of this  seems challenging.

\subsection{The optimal   control problem}
\label{sec:auxil_opt_cont_pb_2}
For $t\le T$, let us consider the optimal control problem
\begin{equation}
  \label{eq:21}
  \phi(t,x,v)=\inf \frac 1 q \int_t^T |\xi''(s)|^q ds,
\end{equation}
where the infimum  is taken on the trajectories $\xi\in W^{2,q}(0,T)$  such that $\xi(t)=x$ and $\xi'(t)=v$ and
$\xi(s)\in(-\infty,0]$ for all $s\in [t,T]$.  

  The dynamic programming principle and elementary considerations yield that
  \begin{eqnarray}
    \label{eq:HJB}
  \quad \quad   - \frac {\partial \phi}{\partial t} (t, x,v)    -v \frac {\partial \phi} {\partial x} (t, x,v) + \frac 1 p \left|  \frac {\partial \phi} {\partial v} (t,x,v) \right|^{p} =0,
        \;
    0\le t<T, x<0, v\in \R,\\
    \phi(T,x,v)=0,  \quad  x<0, v\in \R,\\
    \lim_{x \to 0_-}  \phi(t,x,v)=+\infty,\quad  0\le t<T,  v>0,\\
     \lim_{v \to +\infty}  \phi(t,x,v)=+\infty,\quad  0\le t<T,  x<0,\\
    \phi(t,x,v) =0, \quad  0\le t\le T, x<0, v\le 0,
  \end{eqnarray}
  where $p={q}/{(q-1)}$ and the Hamilton-Jacobi equation is meant in the viscosity sense.

  In Theorem \ref{val_func_char} below, we see that there is an explicit formula for  $\phi$.
  \begin{theorem}\label{val_func_char}
    Let $\phi$ be the value function of the optimal control problem described at the beginning of Section \ref{sec:auxil_opt_cont_pb_2},
    \begin{eqnarray}
       \label{eq:phi=0}
      \phi(t,x,v) =0,  &\hbox{ if }& v (T-t) \le  |x|,\\
           \label{eq:phi_intermediate region}
        \phantom{aa} \phi(t,x,v) = \frac { (2q-1)^{q-1}} {q(q-1)^{q-1}} \frac { \left(v- \frac {|x|}{T-t} \right) ^{q}}{(T-t)^{q-1}},
      &\hbox{ if }&  |x| \le v(T-t)\le   \frac {(2q-1)|x|}{(q-1)}   ,\\
       \label{eq:phi=sing_part}
       \phi(t,x,v) =\frac {q^{q-1}} {(2q-1)^q}   \frac {v^{2q-1}} { |x|^{q-1}},  &\hbox{ if }&  v(T-t)\ge  \frac {(2q-1)|x|}{(q-1)}>0.
    \end{eqnarray}
    The function $\phi$ is of class $C^1$ in $[0,T]\times (-\infty,0)\times \R$, so it satisfies (\ref{eq:HJB}) pointwise in this set.

    Finally, for each $(t,x,v)\in [0,T]\times \R_-\times \R$, the  optimal control problem
   \eqref{eq:21} has a unique minimizer.
  \end{theorem}

  \begin{proof}
It is  clear that  \eqref{eq:phi=0} holds, because if $ v (T-t) \le  |x|$,
  the trajectory given by $\xi (s) = x+vs$, $\xi'(s)=v$ for $s\in [t, T]$ is admissible and achieves the smallest possible cost, namely $0$.

 If $v\ge 0$, any optimal trajectory $\xi$ with $\xi(t)=x$ and $\xi'(t)=v$ is such  that $\xi'$ remains nonnegative. Indeed, if the velocity vanishes at some time $\bar t \in [t, T]$, then  $\xi'(s)=0$ for all $s\in [\bar t , T]$, because any other choice of $\xi'_{|(\bar t, T]}$  would lead to a larger cost.

 Let us consider $(x,v)$ such that $ v (T-t) > |x|$, and an  optimal trajectory  $\xi $ such that $\xi(t)=x$ and $\xi'(t)=v$. Let us set $w=\xi'(T)$. 

 We have just observed that $w\ge 0$. On the other hand, $\min_{t\le s\le T} \xi'(s)<v$, otherwise the trajectory would exit the domain $x\le 0$ before $T$. If $w$ was greater than $\min_{t\le s\le T} \xi'(s)$, then there would exist $\bar t\in (t, T)$ such that  $\xi'(\bar t) = \min_{t\le s\le T} \xi'(s)$. The competitor $\tilde \xi$ :
 $\tilde \xi(s)=x+\int_t^s \tilde \eta(\tau) d\tau$,
 $\tilde \eta(s)= \xi'(s)$ for $t\le s \le \bar t$, $\tilde \eta(s)=\xi'(\bar t)$ for $\bar t \le s\le T$ would be admissible and cost less than $\xi$, which is impossible.  Hence, $w=\min_{t\le s\le T}  \xi'(s)$ and there holds $\phi(t,x,v)= I(T-t, w)$ with $I$  defined in  \eqref{eq:15}.

 Suppose now that $v(T-t)\ge  \frac {(2q-1)|x|}{(q-1)}$. This implies that  $T-t\ge  \frac {(2q-1)|x|}{(q-1)v+q w}$ because $w\ge 0$. Hence, from  \eqref{eq:16}, $I(T-t, w)=\frac {q^{q-1}} {(2q-1)^q}   \frac {(v-w)^{2q-1}} { (|x|-(T-t) w)^{q-1}}$, which is minimal for $w=0$.  We have proved  \eqref{eq:phi=sing_part}.

 There remains to characterize the value function $\phi$ in the region
 \begin{displaymath}
 K=  \left \{|x|< v(T-t) <  \frac {(2q-1)|x|}{(q-1)}\right \}.   
 \end{displaymath}
 If $(t,x,v)\in K$, then there exists a unique $\hat w: 0<\hat w<v$ such that $T-t=  \frac {(2q-1)|x|}{(q-1)v+q \hat w}$. From Lemma  \ref{lem_var_w}, $\hat w $ is the unique minimizer of the function $[0,v)\ni w\mapsto I(T-t, w)$.
Hence, if $\frac {|x|} v< T-t <  \frac {(2q-1)|x|}{(q-1)v}$, then  an optimal  trajectory $\xi$ must satisfy $\xi'(T)=\hat w$, and
its cost must be $I(T-t, \hat w)$. From the last point in Lemma  \ref{lem_var_w}, this yields  \eqref{eq:phi_intermediate region}.

 Since in $K$, $ \phi(t,x,v)= I(T-t, \hat w)$ with $\hat w$ such that
  $T-t=  \frac {(2q-1)|x|}{(q-1)v+q \hat w}$, it is clear that $ \phi$ is continuous across 
  $\left \{ v(T-t) =  \frac {(2q-1)|x|}{(q-1)}\right \}$, where $\hat w=0$.
  It is also clear that $\phi(t,x,v)=0 $ on 
  $\left \{ v(T-t) = |x|\right\}$, where $\hat w=v$, so $\phi$ is continuous across this manifold.

  To prove that  $\phi$ is $C^1$, 
we just need to focus on the boundary of $K$.

First, it is easy to check that on  $\left \{ v(T-t) = |x|\right\}$, all the first order partial derivatives of $ \phi$ exist and take the value $0$.

Second,  it can be checked that on both sides of  $\left \{ v(T-t) =  \frac {(2q-1)|x|}{(q-1)}\right \}$,
\begin{itemize}
\item the values of  $ \frac {\partial  \phi}{\partial t} (t,x,v) $
 coincide  and are equal to $0$
\item The values of $ \frac {\partial  \phi}{\partial v}(t,x,v)$  coincide  and are equal to $ \frac {q  ^{q-1} }   {(2q-1)^{ q-1}} \frac {v^{2q-2} } { |x|^{q-1} }$
\item   The values of $\frac {\partial  \phi}{\partial x}(t,x,v)$ coincide and are equal to
  $ \frac {  (q-1)q^{q-1}} {(2q-1)^q} \frac {v^{2q-1}}{|x|^q}$.
\end{itemize}
We have proven that $\phi$ is $C^1$   in $[0,T]\times (-\infty,0)\times \R$.

Finally, because $\phi(t,x,v)$ is characterized in terms of $I(T-t, w)$ for a unique $w$, the proof of Proposition  \ref{sec:prop_1} yields the uniqueness and the 
characterization of the optimal trajectory related to $(t,x,v)$.
\end{proof}
An easy corollary of Theorem \ref{val_func_char} deals with the value function
of the  optimal control problem \eqref{eq:6} for $n=1$, $\Omega=(-1,1)$, $\ell=0$ and $g=0$:
\begin{corollary} \label{val_func_char_2}
  For $(t,x,v)\in [0,T]\times [-1,1]\times \R$, consider the optimal control problem:
  \begin{equation}
    \label{eq:tilde_phi}
    \tilde\phi(t,x,v)=\inf \frac 1 q \int_t^T |\xi''(s)|^q ds,
  \end{equation}
  where the infimum is taken on the trajectories $\xi\in W^{2,q}(t,T)$ such that $\xi(t)=x$, $\xi'(t)=v$, and $\xi(s)\in [-1,1]$ for all $s\in [t,T]$.
  There holds
  \begin{equation}  \label{eq:tilde_phi_2}
    \tilde  \phi(t,x,v) =\left\{
\begin{array}[c]{rcl}
    \frac {q^{q-1}} {(2q-1)^q}   \frac {|v|^{2q-1}} { (1+x)^{q-1}}, \quad &\hbox{if }&     v(T-t)\le    - \frac {2q-1}{q-1}(1+x)<0,
                           \\ 
 \frac 1 q \frac {(2q-1)^{q-1}} {(q-1)^{q-1}} \frac { \left|v+ \frac {1+x}{T-t} \right| ^{q}}{(T-t)^{q-1}},
      \quad &\hbox{if }&           -\frac {2q-1}{q-1}(1+x)  \le     v(T-t)\le - (1+x)    , \\
 0, \quad &\hbox{if }& -1-x \le v (T-t) \le  1-x,\\ 
  \frac 1 q \frac {(2q-1)^{q-1}} {(q-1)^{q-1}} \frac { \left(v+ \frac {x-1}{T-t} \right) ^{q}}{(T-t)^{q-1}},
      \quad &\hbox{if }&   1-x  \le  v( T-t)\le   \frac {2q-1}{q-1}(1-x)   ,\\ 
      \frac {q^{q-1}} {(2q-1)^q}   \frac {v^{2q-1}} { (1-x)^{q-1}}, \quad &\hbox{if }&  v(T-t)\ge  \frac {2q-1}{q-1}(1-x)> 0,     
\end{array}
\right.
  \end{equation}
  and $\tilde \phi$ is $C^1$ in $[0,T]\times(-1,1)\times \R$. The
restrictions of $\tilde \phi$ to  $[0,T)\times [-1,1)\times [0,+\infty) $ and to  $[0,T)\times (-1,1]\times (-\infty,0]$ are 
 pointwise 
  solutions of  \eqref{eq:HJB}.
\end{corollary}

\section{The state constrained  optimal control problem  \eqref{eq:6} in the case $n=1$}
\label{sec:sc_pb}

We now study the optimal control problem   \eqref{eq:6}  in the mono-dimensional case.
Recall that  $\Omega=(-1,1)$, $\Xi=[-1,1] \times \R$ and $ \Xi^{\rm{ad}}= [-1,1)\times [0,+\infty) \cup  (-1,1]\times (-\infty,0]$.
  
\subsection{Properties of the value function}

\begin{proposition}
  \label{sec:finite-horizon-state}
Let us set $\underline c=\inf_{(y,w)\in \Xi} g(y,w)$, $ \overline c=\sup_{(y,w)\in \Xi} g(y,w)$,
$\underline L= \inf_{(y,w,s)\in \Xi\times [0,T]} \ell(y,w,s)$ and $\overline L= \sup_{(y,w,s)\in \Xi\times [0,T]} \ell(y,w,s)$.

The value function $u$ can be bounded as follows:
\begin{equation}
  \label{eq:24}
\tilde \phi(t,x,v) + \underline L(T-t) +\underline c \le u(t,x,v)\le \tilde \phi(t,x,v) + \overline L(T-t) +\overline c  ,
\end{equation}
where  $\tilde \phi$ is the function appearing in Corollary~\ref{val_func_char_2}.
\end{proposition}
\begin{proof}
It is clear that $(t,x,v)\mapsto \tilde \phi(t,x,v)+ \overline c+ \overline L  (T-t)$ is the value  of the optimal control problem consisting of minimizing the 
cost $ \int_t^T (\frac 1 q|\xi''(s)|^q+\overline L) ds +   \overline c$ on the trajectories $\xi \in W^{2,q}(t,T)$ such that $-1\le \xi(s)\le 1$ for all times $s$  and that $\xi(t)=x$, $\xi'(t)=v$. 
\\
If $\xi$ is an optimal trajectory for this latter problem (we know that it exists), it is clearly admissible and suboptimal for the optimal control problem introduced in the beginning of Section~\ref{sec:sc_pb}.
Hence,
\begin{equation*}
  \begin{split}
  u(t,x,v)& \le  \frac 1 q \int_t^T |\xi''(s)|^q ds +\int_t^T \ell( \xi(s), \xi'(s),s) ds  +g(
\xi(T),\xi'(T) ) \\
  &\le  \frac 1 q \int_t^T |\xi''(s)|^q ds  + (T-t) \overline L + \overline c
\le  \tilde \phi(t,x,v)+ (T-t) \overline L + \overline c,
  \end{split} 
\end{equation*}
and we have proved  the right-hand side of (\ref{eq:24}).

Similarly $(t,x,v)\mapsto \tilde \phi(t,x,v)+ \underline c+ \underline L  (T-t)$ is the value  of the optimal control problem consisting of minimizing the 
cost $ \int_t^T (\frac 1 q|\xi''(s)|^q+\underline L) ds +   \underline c$ on the trajectories $\xi \in W^{2,q}(t,T)$ such that $-1\le \xi(s)\le 1$ for all times $s$  and that $\xi(t)=x$, $\xi'(t)=v$. 
\\
If $\xi$ is an $\epsilon$-optimal trajectory for for the optimal control problem introduced in the beginning of Section~\ref{sec:sc_pb},
 it is  an admissible competitor for the latter problem, hence
\begin{equation*}
  \begin{split}
 \tilde \phi(t,x,v)+ (T-t) \underline L + \underline c &\le   \frac 1 q \int_t^T |\xi''(s)|^q ds  + (T-t) \underline L + \underline c
\\ &\le  \frac 1 q \int_t^T |\xi''(s)|^q ds +\int_t^T \ell( \xi(s), \xi'(s),s) ds  +g(
\xi(T),\xi'(T) )
\\ &\le  u(t,x,v)+\epsilon,
  \end{split} 
\end{equation*}
and the left-hand side of (\ref{eq:24}) is obtained by letting $\epsilon$ tend to $0$.
\end{proof}
The following proposition contains information on the behaviour of $u(t,x,v)$ for a fixed value of $t<T$, as
$x\to 1_-$ and $\frac {v(T-t)}{1-x}\to +\infty$.
\begin{proposition}
  \label{sec:finite-horizon-state-1}
  \begin{enumerate}
  \item   For a fixed time $t$,  $0\le t<T$,  
    \begin{displaymath}
 u(t, x,v)\sim \frac {q^{q-1}} {(2q-1)^q}   \frac { v ^{2q-1}} { (1-x)   ^{q-1}}   
    \end{displaymath}
as $x\to 1_-$  and   $ \frac { v ^{2q-1}} { (1-x)    ^{q-1}}\to \infty$ 
(note that this may occur even if $v\to 0_+$ 
and that this implies $\frac {v(T-t)}{1-x}\to +\infty$).
\item  For a fixed time $t$,  $0\le t<T$, and fixed positive number $C$,
  if $x\to 1_-$,  $\frac { v ^{2q-1}} { (1-x)    ^{q-1}} < C$  and  $\frac {v(T-t)}{1-x}\to +\infty$
(note that this  implies that $v\to 0$ and that this situation occurs if  $ c< \frac { v ^{2q-1}} { (1-x)    ^{q-1}} < C$ for a given constant $c<C$),
   then
 \begin{displaymath}
    u(t, x,v)\le \frac {q^{q-1}} {(2q-1)^q}   \frac { v ^{2q-1}} { (1-x)    ^{q-1}} +u(t,1,0) +o(1),    
   \end{displaymath}
where $o(1)$ is a  quantity that tends to $0$.
 \end{enumerate}
\end{proposition}
\begin{proof}
  To prove the first point, it is enough to combine (\ref{eq:24}) and
  (\ref{eq:tilde_phi_2}) in the case when $v(T-t)\ge  \frac {2q-1}{q-1}(1-x)> 0$.

Let us prove the second point: 
consider sequences $(x_n)_{n\in \N} $ and  $(v_n)_{n\in \N} $ such that  $x_n\to 1_-$,  $ \frac { v_n ^{2q-1}} { (1-x_n)^{q-1}} < C$  and  $\frac {v_n(T-t)}{1-x_n}\to +\infty$.
Consider the function $\eta$
  constructed in the proof of Proposition \ref{sec:prop_1} when $T$ is replaced with $T-t$,
    $x=x_n-1$,  $w=0$,  $\theta_n=  \frac {(2q-1)(1-x_n)}{(q-1)v_n}$.
  Let us set $\tilde t_n = t+\theta_n$,  $\tilde \eta_n(s)= \eta(s-t)$ and $\tilde \xi_n(s)= x_n +\int_t^s \tilde \eta_n(\tau) d\tau$ for $s\in [t, \tilde t_n]$. It is clear that
  $(\tilde \xi_n (t), \tilde \eta_n(t))= (x_n, v_n)$, that $\tilde \eta_n(\tilde t_n)=0$,  $\tilde \xi_n (\tilde t_n)=1$ and that $\frac 1 q \int_t^{\tilde t_n} |\tilde \eta'(\tau)|^q d\tau=   \frac {q^{q-1}} {(2q-1)^q}   \frac { v ^{2q-1}} { (1-x_n)   ^{q-1}}$. Observe also that
$\lim_{n\to \infty }  \tilde t_n =t$.
It is then possible to extend $\tilde \eta_n$ and $\tilde \xi_n$ to $[\tilde t_n, T]$ by setting
$\tilde \xi_n(\tau)=\xi( \tau -\tilde t_n +t)$ where $\xi$ is  an optimal trajectory for $u(t, 1,0)$.
We deduce that the cost associated to the trajectory $\tilde \xi_n$ is $\frac {q^{q-1}} {(2q-1)^q}   \frac { v_n ^{2q-1}} { (1-x_n)   ^{q-1}} +u(t,1,0) + o(1)$, and therefore
  \begin{equation}
  \label{eq:25}
u(t,x_n,v_n)\le
 \frac {q^{q-1}} {(2q-1)^q}   \frac { v_n ^{2q-1}} { (1-x_n)    ^{q-1}} +u(t,1,0) + o(1).
\end{equation}
 \end{proof}

\subsection{Closed graph properties}\label{closed_graph_1}
The closed graph property that follows plays an important role in the theory of relaxed equilibria in related MFGs:
\begin{proposition}\label{prop_closed_graph_1d}
 The graph of the multi-valued map
\begin{displaymath}
  \begin{array}[c]{rl}
    \Gamma^{\rm{opt}}: \; &  \Xi^{\rm ad}  \rightrightarrows  \Gamma,
    \\
    &    (x,v) \mapsto \Gamma^{\rm {opt}}[x,v]     
  \end{array}
\end{displaymath}
is closed, which means: for  any sequence $(x_i, v_i)_{i\in \N}$ such that for all $i\in \N$,   $(x_i, v_i)\in \Xi^{\rm ad}$ with $(x_i, v_i)\to (x,v)\in  \Xi^{\rm ad}$ as $i\to \infty$, consider a sequence     $(\xi^i)_{i\in \N}$ such that for all $i\in \N$, $(\xi^i) \in \Gamma^{\rm{opt}}[x_i, v_i]$; if $(\xi_i,\xi_i')$ tends to $(\xi,\xi')$ uniformly,  then $\xi\in \Gamma^{\rm{opt}}[x, v]$.
\end{proposition}
\begin{remark}
  It is worth noticing that in  Proposition \ref{prop_closed_graph_1d},
  even if $\|\xi^{i}\|_{W^{2,q}(0,T)}$ blows up as $i$ tends to $\infty$, the limiting trajectory (convergence in $C^1$), $\xi$, is optimal for $u(0,1,0)$ which, in particular,  implies that $\|\xi\|_{W^{2,q}(0,T)}<\infty$. Besides, it is precisely the fact that  $\|\xi^{i}\|_{W^{2,q}(0,T)}$ may blow up that brings difficulties in the proof.
\end{remark}

\begin{proof}
  This convergence stated in Proposition \ref{prop_closed_graph_1d}  has already been proved in \cite{MR4444572} in the three situations that follow
  \begin{enumerate}
  \item  $(x,v)\not=(-1,0)$ and $(x,v)\not=(1,0)$
  \item $(x,v)=(1,0)$ and 
    $\lim_{i\to \infty}\frac {(v_i)_+^{2q-1} }{|x_i-1|^{q-1}}=0$,
    with the convention that if $x_i=1$ and  $(v_i)_+=0$, then the quotient $\frac {(v_i)_+^{2q-1} }{|x_i-1|^{q-1}}$ takes the value $0$
  \item $(x,v)=(-1,0)$ and $\lim_{i\to \infty}\frac {(v_i)_-^{2q-1} }{|x_i+1|^{q-1}}=0$,
    with the convention that if $x_i=-1$ and  $(v_i)_-=0$, then  the quotient $\frac {(v_i)_-^{2q-1} }{|x_i-1|^{q-1}}$ takes the value $0$,
  \end{enumerate}
so we refer to  \cite[Proposition 3.9]{MR4444572} for these cases, in which $(\xi^i)_i$ is bounded in $W^{2,q}(0,T)$.

There remains to study the following two situations:
\begin{enumerate}
\item  $(x,v)=(1,0)$ and $\frac {(v_i)_+^{2q-1} }{|x_i-1|^{q-1}}\ge c$ for a given positive constant $c$
\item $(x,v)=(-1,0)$ and $\frac {(v_i)_-^{2q-1} }{|x_i+1|^{q-1}}\ge c$ for a given positive constant $c$
\end{enumerate}
Let us focus on the former since the latter can be addressed in the same manner.

Assume that $\frac {(v_i)_+^{2q-1} }{|x_i-1|^{q-1}}\ge c$ and that $(\xi_i,\xi_i')$ tends to $(\xi,\xi')$ uniformly, with $\xi(0)=1$ and $\xi'(0)=0$. Note that  $|v_i|/|x_i-1|$ tends to $+\infty$.  We make out two cases:
\begin{description}[style=unboxed,leftmargin=0cm]
\item[Case 1:] there exists $t>0$ such that $\xi(t)\not =1$.\\ Let $\theta_1$ be the largest real number such that $\xi(t)=1$ for all $t\in [0,\theta_1]$. There exists $\theta_2>\theta_1$ such that  $\xi(t)<1$ for $t\in (\theta_1, \theta_2]$ and for all $\epsilon>0$, there exists $\theta_{3,\epsilon}\in (\theta_1, \theta_1+\epsilon)$ such that $\xi'(\theta_{3,\epsilon})<0$. 
  Let $t_i\in (0,T)$ be the smallest time at which $\xi_i'$ vanishes. Since $\xi_i'$ tends to $\xi'$ uniformly, we know that  $t_i$ exists for $i$ large enough. After the extraction of a subsequence, we may suppose that $t_i$ converges to some $\bar t\in [0,T]$.
The uniform convergence of $\xi_i'$ implies that $\bar t \le \theta_1$.

Arguing   as in the proof of the second point in Proposition~\ref{sec:finite-horizon-state-1},
we find that
  \begin{equation}
    \label{eq:26}
J( \xi_i)= u(0,x_i,v_i)\le    \frac {q^{q-1}} {(2q-1)^q}   \frac { v_i ^{2q-1}} { (1-x_i)    ^{q-1}} +u(0,1,0) + o(1).
  \end{equation}
  Next, we know from Lemma \ref{lemma:ass_infimum0} that
  \begin{equation}
    \label{eq:27}
    \frac 1 q \int_0^{t_i} |\xi_i''(s)|^q ds \ge   \frac {q^{q-1}} {(2q-1)^q}   \frac { v_i ^{2q-1}} { (1-x_i)    ^{q-1}}.
  \end{equation}
 From the assumption on the uniform convergence of $(\xi_i, \xi_i')$ to $(\xi, \xi')$, we easily deduce that 
  \begin{equation}
    \label{eq:28}
    \begin{split}
&    \lim_{i\to \infty}  \int_0^T \ell(\xi_i(s),\xi_i'(s) ,s) 
  ds+g(\xi_i(T), \xi_i'(T)) \\ = &\int_0^T \ell(\xi(s),\xi'(s) ,s) 
  ds+g(\xi(T), \xi'(T)).     
    \end{split}
  \end{equation}
  Combining \eqref{eq:26}, \eqref{eq:27} and \eqref{eq:28} yields
  \begin{equation}
    \label{eq:29}
 \frac 1 q \int_{t_i}^T |\xi_i''(s)|^q ds
  \le  u(0,1,0) -  \int_0^T \ell(\xi(s),\xi'(s) ,s)  
ds- g(\xi(T), \xi'(T))+ o(1).
\end{equation}
Hence $\one_{[t_i,T]} \xi_i''$ is bounded in $L^q(0,T)$.
Since $\xi_i'(t_i)= 0$, $\one_{[t_i,T]} \xi_i''$ is the weak derivative of  $\one_{[t_i,T]} \xi_i'$, which is itself the weak derivative of $\tilde \xi_i= \xi_i(t_i)  \one_{[0,t_i)}+  \one_{[t_i,T]} \xi_i$. Note that $\lim_{i\to \infty} \xi_i(t_i)=0$, because $\bar t \le \theta_1$. Therefore, up to  the extraction of a subsequence, we may assume that $\tilde  \xi_i$ converges weakly in $W^{2,q}(0,T)$. The weak limit can be nothing  but $\xi$. Hence $\xi\in   W^{2,q}(0,T)$. 

Moreover
\begin{equation}
    \label{eq:30}
\int_{0}^T |\xi''(s)|^q ds  \le \liminf_{i\to \infty}  \int_{0}^T |\tilde \xi_i''(s)|^q ds=   \liminf_{i\to \infty}  \int_{t_i}^T |\xi_i''(s)|^q ds.  
\end{equation}
We deduce from  \eqref{eq:29} and \eqref{eq:30} that
\begin{equation}
   \label{eq:31}
\frac 1 q   \int_{0}^T |\xi''(s)|^q ds + \int_0^T \ell(\xi(s),\xi'(s) ,s) 
  ds+ g(\xi(T),\xi'(T))\le u(0,1,0),
\end{equation}
which implies that  $ \xi\in \Gamma^{\rm{opt}}[1, 0]$.
  
\item [Case 2:] $\xi(t)=1$ for all $t\in [0,T]$.\\
  If after the extraction a subsequence, there exists $t_i$ for all $i$ such that $\xi_i'(t_i)=0$, then we can reproduce the arguments in case 1 and conclude.

  Let us now suppose that for $i$ large enough, $\xi_i'$ is positive on $[0,T]$.
We know that $\lim_{i\to \infty}  \int_0^T \ell(\xi_i(s),\xi_i'(s) ,s) 
  ds+ g(\xi_i(T), \xi_i'(T))=  \int_0^T \ell(1,0 ,s) 
 ds+ g(1,0)$. 
There exists  $t_i\in (0,T)$ such that $\xi_i'(t_i)=  2(1-x_i)/ T$, otherwise the trajectory would exit the domain before $T$.
 Let us define the modified  trajectory $\tilde \xi_i\in C^1$ as follows:
 \begin{eqnarray}
   \tilde \xi_i (t)= \xi_i(t) \quad \hbox{if } t\le t_i,\\
   \tilde \xi_i'(t)= \min( \xi_i'(t), 2(1-x_i)/ T)  \quad \hbox{if } t \ge t_i.
 \end{eqnarray}
 The modified trajectory is admissible and corresponds to the same initial condition $(x_i,v_i)$, and both  $(\xi_i,\xi'_i)$ and  $(\tilde \xi_i,\tilde\xi'_i)$ tend to $(1,0)$ uniformly in $[0,T]$, and  $\int_0^T  |\tilde \xi_i''(t)|^q dt\le \int_0^T  | \xi_i''(t)|^q dt$.
 Therefore,
 \begin{equation}
   \label{eq:32}
   \begin{split}
     &u(0,x_i,v_i)\\ \le &\frac 1  q \int_0^T  |\tilde \xi_i''(t)|^q dt
+
    \int_0^T \ell(\tilde \xi_i(s),\tilde \xi_i'(s) ,s) 
   ds+ g(\tilde \xi_i(T),\tilde  \xi_i'(T))
\\
     \le&\frac 1  q \int_0^T  | \xi_i''(t)|^q dt +
    \int_0^T \ell(\tilde \xi_i(s),\tilde \xi_i'(s) ,s) 
  ds+ g(\tilde \xi_i(T),\tilde  \xi_i'(T))\\
  \le &u(0,x_i,v_i)+o(1).
   \end{split}
 \end{equation}

 By taking $x=x_i-1$, $\theta=T$, $v=v_i$ and $w=\tilde \xi'_i(T)\le  2(1-x_i)/ T    \ll v_i$ in  Proposition \ref{sec:prop_1}, we see that
  \begin{displaymath}
    \frac 1 q \int_0^T |\tilde \xi_i''(s)|^q ds \ge \frac {q^{q-1}} {(2q-1)^q}   \frac {v_i^{2q-1}} { |1-x_i|^{q-1}}+o(1).
  \end{displaymath}
  Therefore
  \begin{equation}
    \label{eq:33}
     u(0,x_i,v_i) \ge \frac {q^{q-1}} {(2q-1)^q}   \frac {v_i^{2q-1}} { |1-x_i|^{q-1}}+
 \int_0^T \ell(1,0 ,s) 
  ds+ g(1,0)+o(1).
  \end{equation}
 We deduce from  \eqref{eq:26} and  \eqref{eq:33} that $u(0,1,0)= \int_0^T \ell(1,0 ,s) 
  ds+ g(1,0)$, which  implies that  $ \xi\in \Gamma^{\rm{opt}}[1, 0]$.
\end{description}
\end{proof}

\section{ The state constrained  optimal control problem  \eqref{eq:6} in the case $n>1$ and $q=2$}
\label{sec:multid}

Here, we focus on running costs which are quadratic in the acceleration because some  arguments  will require the sharp estimate  \eqref{eq:20} on $\inf_{\theta} I(\theta,w)$ contained in point (ii) in Proposition \ref{lemma:ass_infimum}.

\paragraph{\bf Preliminary facts}
  Lemma  \ref{def:local_charts} below is devoted to local charts that will enable us to rely on the results obtained in Section \ref{sec:auxil_opt_cont_pb}  for one-dimensional problems. We give its proof for completeness although the main ideas come from \cite[Section 14.6]{MR1814364}.
  \begin{lemma} \label{def:local_charts}
    For all $x\in \partial \Omega$, there exist an open neighborhood $V_x$ of $x$ in $\R^n$, a positive number $R_x$ and a $C^2$-diffeomorphism $\Phi_x$ from $V_x$ onto $B(0, R_x)$ with the following properties:
  \begin{enumerate}
  \item $\Phi_x(x)=0$ and for all $z\in V_x$, the $n^{\rm{th}}$ coordinate of $\Phi_x(z)$ is $d(z)$, i.e. $e_n\cdot \Phi_x(z)= d(z)$, where $e_n$ is the $n^{\rm{th}}$ vector of the canonical basis.
  \item $ \Phi_x|_{V_x\cap \Omega}$ is a $C^2$-diffeomorphism  from $V_x\cap \Omega$ onto $ B_-(0, R_x)=B(0, R_x)\cap \{x_n<0\}$, and $ \Phi_x|_{V_x\cap \partial\Omega}$ is a $C^2$-diffeomorphism  from $V_x\cap \partial \Omega$ onto $B(0, R_x)\cap \{x_n=0\}$.
  \item The inverse  $\Psi_x$ of $\Phi_x$  is a $C^2$-diffeomorphism  from $B(0, R_x)$ onto $V_x$ and $\Psi_x|_{ \{y\in B(0, R_x); y_n=0\}} $ is a $C^2$-diffeomorphism onto $V_x\cap \partial \Omega$.
    \item For all $y\in  B(0, R_x)$ such that $y_n=0$,
 \begin{equation}\label{eq:37}
  D\Psi_x(y+te_n) e_n=  n(\Psi_x(y)).   
 \end{equation}     
 Also,
\begin{equation}
  \label{eq:38}
\nabla d(z)=D\Phi_x^T (z) e_n, \quad \hbox{for all } z\in V_x.
\end{equation}
 In particular, $n(x)=D\Phi_x^T (x) e_n$.
  \end{enumerate}
  \end{lemma}

  \begin{proof}
  For $h>0$, set $\omega_h=\{z \in \R^n , |d(z )| < h\}$.
    It is proven in \cite[Section 14.6]{MR1814364} that there exists  $\mu>0$ smaller than the inverse of all principal curvatures of $\partial \Omega$ at all points of $\partial \Omega$,
  such  for all $z\in \omega_\mu$, there exists a unique point $\pi(z)\in \partial \Omega $ such that $|d(z)|=|z-\pi(z)| $. The points $z$ and $\pi(z)$ are related by
  \begin{displaymath}
    z= \pi(z)+ d(z) n(\pi(z)).
  \end{displaymath}
  
  Given  $x\in \partial \Omega$, it is always possible to rotate the coordinates in such a way that the  coordinate  $z_n$ of a  point $z\in \R^ n$ lies in the direction of $n(x)$. Let us define $z'\in \R^{n-1}$ by $z=(z', z_n)$. The tangent subspace to $\partial \Omega$ at $x$  is therefore $T_x(\partial \Omega)= \{ z\in \R^n \;:\; z_n=0\}$. From the assumptions, there exists an open  neighborhood $\cN_x$ of $x$ and  $\varphi \in C^3(  T_x(\partial \Omega);\R  ) $ such that $\varphi(0) =0$, $D\varphi(0)=0$, and for all $z\in \cN_x\cap \partial \Omega$,   $ z_n -x_n= \varphi(z'-x')$.
\\  
Let us set $\cN'_x= (\cN_x -x) \cap  T_x(\partial \Omega)$ and define
$\Psi_x: \cN'_x\times \R  \to \R^n$ by $\Psi_x (y', d) = y+ d n(y)$ with $y=(x'+ y',x_n+ \varphi(y'))$.  The map $\Psi_x $ is of class $C^2$, and its Jacobian matrix at $(0, d)$ is given by $D\Psi_x(0, d) = \rm{diag}(1-\kappa_1 d, \cdots, 1-\kappa_{n-1} d, 1)$, where $\kappa_i$ are the principal curvatures of $\partial \Omega$ at $x$. The Jacobian of $\Psi_x$ at $(0, d) $ is therefore $\prod_{i=1}^{n-1} (1-\kappa_i d)>0$ if $|d|<\mu$. It follows from the inverse mapping theorem that there exists a neighborhood $\cM_x$ of $x$ such that the map $z\mapsto (y'(z), d(z))$ is  a  $C^2$-diffeomorphism from $\cM_x$ to an open neighborhood of $0_{\R^n}$ contained in $ \cN'_x\times \R$. We deduce that  $\pi(z)=(x'+y'(z), x_n+ \varphi(y'(z)))$ for $z\in \cM_x$.
Hence, $z\mapsto \pi(z)$ is also in $C^2(\cM_x)$. This implies that $\nabla d (z)= n(\pi(z))$ is of class $C^2$, hence $d\in C^3(\cM_x)$.

Note that by a slight modification of the previous argument,  see  \cite[Section 14.6]{MR1814364}, one can actually prove that
$\pi\in C^2(\omega_\mu)$ and $d\in C^3(\omega_\mu)$ for $\mu$ defined above, but we shall not need this result.

 Finally, for $x\in \partial \Omega$,  it is possible to choose a positive radius $R_x$ and
 an open neighborhood   $V_x\subset \cM_x$ of $x$  such that $\Psi_x$ defines a $C^2$-diffeomorphism from $B(0,R_x)$ onto $V_x$, and denoting by  $\Phi_x$ its inverse, $\Phi_x$ and $\Psi_x$ satisfy points 1. to 4.
\end{proof}

 \paragraph{\bf Closed graph properties and singularities of $u$ }
\label{sec:clos-graph-prop}
\begin{proposition}\label{prop:clos-graph-prop}
  The graph of the multi-valued map
\begin{displaymath}
  \begin{array}[c]{rl}
    \Gamma^{\rm{opt}}: \; &\Xi^{\rm{ad}} \rightrightarrows  \Gamma,
    \\
    &    (x,v) \mapsto \Gamma^{\rm {opt}}[x,v]     
  \end{array}
\end{displaymath}
is closed, which means: for  any sequence $(x^i, v^i)_{i\in \N}$ such that for all $i\in \N$,   $(x^i, v^i)\in \Xi^{\rm{ad}} $ with $(x^i, v^i)\to (x,v)$ as $i\to \infty$, consider a sequence     $(\xi^i)_{i\in \N}$ such that for all $i\in \N$, $\xi^i \in \Gamma^{\rm{opt}}[x^i, v^i]$; if $(\xi^i,\frac {d\xi^i} {dt} )$ tends to $(\xi,\frac {d\xi} {dt})$ uniformly,  then $\xi\in \Gamma^{\rm{opt}}[x, v]$.
\end{proposition}

\begin{proof}
  If $x\in \Omega$ or if $x\in \partial \Omega$ and the sequence $(x^i, v^i)_{i\in \N}$ satisfies the further property
  \begin{equation}
  \label{eq:39}
 \lim_{i\to \infty}  \frac {(v^i \cdot \nabla d (x^i) )_+^3} {\left |d(x^i)\right|} = 0,
\end{equation}
with the same convention as in Lemma \ref{sec:bounds_and_continuity},
then the desired result is \cite[Prop. 2.2 and Prop. 3.9]{MR4444572}.

There remains to consider the case when $x\in \partial \Omega$, $v\cdot n(x)=0$ and
\begin{equation}
  \label{eq:40}
  (v^i \cdot \nabla d (x^i))_+^{3} \ge c \left |d(x^i)\right|
\end{equation}
for a given positive constant $c$.

Let $V_x,\, R_x,\, \Phi_x, \, \Psi_x$ be as in Lemma \ref{def:local_charts}. Because $\xi $ is $C^1$ and $\xi^i$ tends to $\xi$ in $C^1$, there exits a time $T_1$, $0<T_1\le T$ such that  all the trajectories $\xi$ and $\xi^i$,$i\in \N$ remain in $V_x$ in $[0,T_1]$. Therefore, the trajectories $\widehat{\xi}= \Phi_x\circ \xi$ and $\widehat{\xi^i}=
\Phi_x\circ \xi^{i}$ are well defined as $C^1$ functions from $[0,T_1]$ to $B_-(0,R_x)$.

Note that
\begin{displaymath}
  \frac {d {\xi}} {dt}(s)= D\Psi_x( \widehat{\xi}(s)) \frac {d \widehat{\xi}} {dt}(s),
\end{displaymath}
and
\begin{equation}
  \label{eq:41}
  \frac {d^2  {\xi}} {dt^2}(s)= D\Psi_x( \widehat{\xi}(s)) \frac {d^2  \widehat{\xi}} {dt^2 }(s)  +  \left( D^2 \Psi_x( \widehat{\xi}(s))  \frac {d \widehat{\xi}} {dt}(s)\right)  \frac {d \widehat{\xi}} {dt}(s) ,
\end{equation}
and that the same formula hold  mutatis mutandis for $\xi^i$.

We  set $\widehat x^i = \Phi_x(x^i)$, $\widehat v^i =  \frac {d \widehat{\xi^i}} {dt}(0)= D\Phi_x( x^i) v^i$, and $\hat x=0$, $\hat v= D\Phi_x(x) v $. Obviously, $\widehat x^i\cdot e_n=d(x^i)$. There also holds that
 $e_n \cdot \widehat v= e_n\cdot D\Phi_x(x) v= n(x)\cdot v=0$ and
\begin{displaymath}
  e_n\cdot \widehat v^i=  e_n\cdot D\Phi_x(x^i) v^i=
  (D\Phi_x(x^i))^T e_n \cdot v_i=\nabla d(x^i) \cdot v^i  = o(1),
\end{displaymath}
because $\nabla d(x^i)=\nabla d(x) +O(|x^i-x|)= n(x)+O(|x^i-x|)$.
Hence, \eqref{eq:40} becomes
\begin{equation}
  \label{eq:42}
   \left(  \widehat v^i \cdot  e_n\right)^{3} \ge c \left |\widehat x^i\cdot e_n \right|.
\end{equation}
In the same manner, $\widehat \xi^i(s)\cdot e_n=d(\xi^i(s))$ and
\begin{displaymath}
  e_n\cdot \frac {d\widehat \xi^i}{dt}(s)=\nabla d(\xi^i(s) )\cdot  \frac {d\xi^i}{dt}(s).
\end{displaymath}

\paragraph{\bf First Step}  Let us build  a competitor for $\xi^i$: recalling that $ \widehat v_i\cdot e_n=v^i \cdot \nabla d (x^i)$, 
set $\theta^i= \frac {3 |d(x^i)|}{ v^i \cdot \nabla d (x^i)}$ which  tends to $0$ as $i\to \infty$,   $\Pi \widehat v_i=   \widehat v_i -  (\widehat v_i\cdot e_n)e_n$,   and construct the real valued function $ s\mapsto \eta^i(s)$ 
as in the proof of Proposition \ref{sec:prop_1} with $T$  replaced with $T_1$, $x=d(x^i)$,   $v= v^i \cdot \nabla d (x^i)=  \widehat v_i\cdot e_n $,  $w=0$,  $\theta=\theta^i$. Set
\begin{displaymath}
\tilde \xi^i (s)=    \Psi_x\left( \widehat x^i +  s\Pi \widehat v_i+   \int_0^s \eta^i(\tau) e_n d\tau\right)  
\end{displaymath}
if $s\le \theta^i$ and extend $\tilde \xi^i $ in $[\theta^i, T]$ by an optimal trajectory for $u\left (\theta^i,    \tilde \xi^i(\theta^i ), \frac {d  \tilde \xi^i} {dt}(\theta^i)\right)$.
Such an optimal trajectory exists  since  $\tilde \xi^i(\theta^i)\in \partial \Omega$ and $\frac {d\tilde \xi_i} {dt}(\theta^i)\cdot n( \tilde \xi^i(\theta^i))=0$, thus $(\xi^i(\theta^i), \frac {d\tilde \xi_i} {dt}(\theta^i))\in \Xi^{\rm ad}$.
Note also that from Lemma \ref{sec:bounds_and_continuity}, $u(\theta_i, \xi^i(\theta^i), \frac {d\tilde \xi_i} {dt}(\theta^i))$ converges to $u(0,x,v)$ as $i\to \infty$.

Set also $\overline \xi^i (s)= \widehat x^i  + s\Pi \widehat v_i+       \int_0^s \eta^i(\tau) e_n d\tau$ if $s\le \theta^i$.
We deduce from Proposition \ref{sec:prop_1} that 
\begin{equation*}
    \ds   \frac 1 2 \int_0^{\theta^i} \left|  \frac {d^2\overline \xi^i}{dt^2}(s)\right|^2  ds = \frac 1 2 \int_0^{\theta^i} \left|e_n \cdot  \frac {d^2\overline \xi^i}{dt^2}(s)\right|^2  ds                                                                                                 = \ds \frac 2 9 \frac {(e_n\cdot \widehat v^i)^3 }{ |e_n\cdot \widehat x^i|}.
\end{equation*}
But
\begin{displaymath}
  \frac {d^2  {\tilde \xi^i}} {dt^2}(s)= D\Psi_x( \overline{\xi}^i(s)) \frac {d^2  \overline{\xi}^i} {dt^2 }(s)  +  \left( D^2 \Psi_x( \overline{\xi}^i(s))  \frac {d \overline{\xi}^i} {dt}(s)\right)  \frac {d \overline{\xi}^i} {dt}(s).
\end{displaymath}
Note that the second term in the right-hand side is bounded and that
\begin{eqnarray*}
   D\Psi_x( \overline{\xi}^i(s)) \frac {d^2  \overline{\xi}^i} {dt^2 }(s)  &=&  \left(D\Psi_x(\widehat x^i + s\Pi \widehat v_i ) + O(\theta^i |e_n\cdot \widehat v^i| ) \right)\frac {d^2  \overline{\xi}^i} {dt^2 }(s)\\  &=&  \left(e_n\cdot\frac { d^2  \overline{\xi}^i} {dt^2 }(s)\right) \left(D\Psi_x(\widehat x^i +s\Pi \widehat v_i) e_n + O(|d(x^i)|) \right)  .
\end{eqnarray*}
Combining all the observation above leads to
  \begin{eqnarray*}
  \ds   \frac 1 2 \int_0^{\theta^i} \left|  \frac {d^2\tilde \xi^i}{dt^2}(s)\right|^2  ds &= & \ds \frac 2 9    \left(|D\Psi_x(\widehat x^i+ s\Pi \widehat v_i) e_n|^2 + O(|d({x}^i)|) \right)  \frac {(e_n\cdot \widehat v^i)^3 }{ |e_n\cdot \widehat x^i|} + o(1)\\ & =&
      \ds \frac 2 9    |D\Psi_x(\widehat x^i+ s\Pi \widehat v_i) e_n|^2  
          \frac {( v^i \cdot \nabla d (x^i))^3 }{ |d(x^i)|} + o(1) .
  \end{eqnarray*} 
  But from \eqref{eq:37},  $|D\Psi_x(\widehat x^i+ s\Pi \widehat v_i) e_n|= 1$,
  therefore
  \begin{equation}\label{eq:43}
       \ds   \frac 1 2 \int_0^{\theta^i} \left|  \frac {d^2\tilde \xi^i}{dt^2}(s)\right|^2  ds
       = \frac 2 9   
          \frac {( v^i \cdot \nabla d (x^i))^3 }{ |d(x^i)|} + o(1) .
  \end{equation}
Hence, 
  \begin{equation}
    \label{eq:44}
    \begin{split}
    u(0,x^i, v^i) \le& \ds  \frac 2 9   
    \frac {( v^i \cdot \nabla d (x^i))^3 }{ |d(x^i)|} + u\left(\theta^i, \tilde \xi^i(\theta^i), \frac {\partial\tilde \xi^i}   {\partial t}(\theta^i)\right)+ o(1)
    \\ =& \ds
 \frac 2 9   
 \frac {( v^i \cdot \nabla d (x^i))^3 }{ |d(x^i)|} + u(0, x, v)+ o(1),
    \end{split}
  \end{equation}
  where the last line is a consequence of  Lemma \ref{sec:bounds_and_continuity}.

  \paragraph{\bf Second Step}

 Since the domain $\Omega$ is convex, $\Omega$ is contained in the  half-space
 $ E= \{ z\in \R^n : (z-x^i)\cdot \nabla d(x^i) < -d(x^i)\} $.

 Set now $\epsilon^i = \frac {|d(x^i)|} {v^i \cdot \nabla d (x^i)}$. From \eqref{eq:40}, it is clear that $ |d(x^i)| \ll \epsilon^i \ll  v^i \cdot \nabla d (x^i)$.

 There  exists $\tau^i\in (0,  v^i \cdot \nabla d (x^i)   )$ such that $\nabla d(x^i)\cdot \frac {d \xi^i}{dt}(\tau^i)= \epsilon^i$, otherwise
 the trajectory $ \xi^i$ would exit $E$ and therefore $\Omega$ in short time.

 From  \eqref{eq:20} in Proposition \ref{lemma:ass_infimum}, point (ii), we deduce that
\begin{equation*}
  \begin{array}[c]{rcl}
    \ds  \frac 1 2 \int_0^{\tau^i} \left| \frac {d^2 \xi^i}{dt^2}(s)\right|^2  ds  & \ge & \ds 
                                                                                            \frac 1 2 \int_0^{\tau^i} \left|\nabla d(x^i)\cdot \frac {d^2 \xi^i}{dt^2}(s)\right|^2  ds
    \\                                                                                 &\ge& \ds  \frac 2 9 \frac {(v^i\cdot \nabla d(x^i))^3}{|d(x^i)|} 
                                                                                         + O\left(  ( v^i\cdot  \nabla d(x^i))^2  \frac{\epsilon^i}   { |d(x^i)|}\right)
    \\
     &=& \ds  \frac 2 9 \frac {(v^i\cdot \nabla d(x^i))^3}{|d(x^i)|}  +o(1).
  \end{array}
\end{equation*}
Note that we have applied  Proposition \ref{lemma:ass_infimum}, point (ii)  to trajectories
$y\in C^2([0,\tau^i ]; \R_-)$ such that $y(0)= d(x^i)$, $ y'(0)=v^i\cdot  \nabla d(x^i) $, and in particular to $y(s)=(\xi^i(s)-x^i)\cdot \nabla d(x^i) +d(x^i)$. 
Therefore, from the  $C^1$ convergence of $\xi^i$ to $\xi$,
\begin{equation}
  \label{eq:45}
  \begin{split}
    u(0, x^i, v^i) \ge & \ds \frac 2 9   
 \frac {( v^i \cdot \nabla d (x^i))^3 }{ |d(x^i)|}  +\frac 1 2 \int_{\tau^i}^T    \left| \frac {d^2 \xi^i}{dt^2}(s)\right|^2  ds \\ &\ds + \int_0^T \ell( \xi(s), \frac {d \xi}{dt}(s),s ) ds + g(\xi(T),\frac {d \xi}{dt}(T) )   +o(1).
  \end{split}
\end{equation}

 \paragraph{\bf Third step}
 Combining \eqref{eq:45} and   \eqref{eq:44} shows that $\one_{[\tau^i,T]} \frac {d^2 \xi^i}{dt^2}$ is bounded in $L^2(0,T)$.
 But $\one_{[\tau^i,T]}  \frac {d^2 \xi^i}{dt^2}$ is the weak derivative of  $\one_{[\tau^i,T]}  \frac {d \xi^i}{dt}+  \one_{[0,\tau^i]}  \frac {d \xi^i}{dt} (\tau^i) $, which is itself the weak derivative of $ \kappa^i   (s)=  \one_{[\tau^i,T]}(s)   \xi^i(s) +\one_{[0,\tau^i)}(s)  (\xi_i(\tau^i) +   \frac {d \xi^i}{dt} (\tau^i) (s-\tau^i)   )$.
Therefore, up to  the extraction a subsequence, we may assume that $\kappa^i$ converges weakly in $W^{2,2}(0,T)$. The weak limit can be nothing  but $\xi$. Hence $\xi\in   W^{2,2}(0,T)$. 

Moreover
\begin{equation}
    \label{eq:46}
\int_{0}^T | \frac {d^2 \xi}{dt^2}(s)|^2 ds  \le \liminf_{i\to \infty}  \int_{0}^T | \frac {d^2 \kappa^i}{dt^2}|^2 ds=   \liminf_{i\to \infty}  \int_{\tau^i}^T |\frac {d^2 \xi^i}{dt^2}(s)|^2 ds.  
\end{equation}
Finally, we deduce from  \eqref{eq:46},  \eqref{eq:45} and   \eqref{eq:44} that
\begin{equation}
   \label{eq:47}
   \frac 1 2   \int_{0}^T |\frac {d^2 \xi}{dt^2}(s)|^2 ds + \int_0^T \ell(\xi(s),\frac {d \xi}{dt}(s),s ) ds + g(\xi(T),\frac {d \xi}{dt}(T) )
\le u(0,x,v),
\end{equation}
which implies that  $ \xi\in \Gamma^{\rm{opt}}[x, v]$.
\end{proof}

\begin{remark}
  The assumption on the convexity of $\Omega$ has been used in the second step of the proof,
  where a sharp bound from below for $u(0,x^i, v^i)$ is obtained, and has allowed us to rely on the results obtained in the one dimensional case. Without this assumption, it seems that
  to bound $u(0,x^i, v^i)$ from below, one should also take into account  the tangential component of the  velocity
  (Coriolis effect)
  and the curvature of $\partial \Omega$, and it  is not clear how to obtain the the desired result.
\end{remark}
 Recall that $u$ is a viscosity solution of \eqref{eq:1000} in $\Omega\times \R^n\times (0,T)$
(with $p=2$) and satisfies \eqref{eq:1001}. It is also clear that $u(x,v,t)=+\infty$ if $t<T$, $x\in \partial \Omega$ and $n(x)\cdot v>0$.
Even if we shall not discuss this aspect, we also expect that $u$ be a supersolution of \eqref{eq:1000} at $(x,v,t)$ such that  $t<T$, $x\in \partial \Omega$ and $n(x)\cdot v<0$. The function $u$ is discontinuous at  $(x,v,t)$ such that  $t<T$, $x\in \partial \Omega$ and $n(x)\cdot v=0$. Proposition  \ref{sec:finite-horizon-state-1_b} that follows contains information on the singularities of $u$ at such points.
\begin{proposition}
  \label{sec:finite-horizon-state-1_b}
   \begin{description}[style=unboxed,leftmargin=0cm]
  \item[1.]   For a fixed time $t$,  $0\le t<T$,  $(y,w)\in \Xi^{\rm{ad}}$,
    \begin{displaymath}
 u(t, y,w)\sim \frac {2} {9}   \frac { (w\cdot \nabla d(y) )_+ ^{3}} { |d(y)| }   
    \end{displaymath}
    as $d(y)\to 0$  and   $  \frac { (w\cdot  \nabla d(y))_+ ^{3}} { |d(y)| }    \to \infty$ 
(note that this may occur even if $(w\cdot  \nabla d(y))_+ \to 0$ 
and that this implies $\frac {(w\cdot  \nabla d(y))_+ (T-t)}{ |d(y)| }\to +\infty$).
\item[2.]  For a fixed time $t$, $0\le t<T$, and a fixed constant $C$, consider  $(x,v)$ such that $x\in \partial \Omega$, $v\cdot n(x)=0$. 
  If $y\to x$, $w\to v$,  $   \frac { (w\cdot  \nabla d(y))_+ ^{3}} { |d(y)| } < C$, and  $\frac {(w\cdot  \nabla d(y))_+ (T-t)}{ |d(y)| }\to +\infty$
(note that the latter assumption  holds if $  c< \frac { (w\cdot  \nabla d(y))_+ ^{3}} { |d(y)| } < C$ for a given positive constant $c<C$),
then
 \begin{displaymath}
   u(t, y,w)\le  \frac {2} {9}   \frac { (w\cdot \nabla d(y) )_+ ^{3}} { |d(y)| }   
 +u(t,x,v) +o(1),    
   \end{displaymath}
where $o(1)$ is a  quantity that tends to $0$.
\end{description}
\end{proposition}
\begin{proof}
  We skip the proof since it is  similar to the first step in the proof of Proposition
  \ref{prop:clos-graph-prop}.
\end{proof}

\section{Application to mean field games with state constraints}
\label{sec:mfg}
We first describe the class of MFGs  that we wish to study,
then address the existence of {\sl relaxed equilibria}.

\paragraph{\bf Setting and definition of relaxed equilibria}

The sets $\Omega$, $\Xi$, $ \Xi^{ \rm{ad}}$, are defined  in Subsection \ref{sec:setting}.
Let $\cP(\Xi)$ be the set of Borel probability measures on $\Xi$.
Let $F, G: \cP(\Xi)\to C_b^0(\Xi;\R)$ be  bounded and continuous maps (the continuity is  with respect to the narrow convergence  in $ \cP(\Xi)$) and $L$ be a continuous and  bounded  function defined on $\Xi \times [0,T]$.

We wish to study MFGs defined as follows:
at equilibrium,  a generic agent
whose state at time $t$ is $(x,v)\in \Xi$ chooses her strategy  by minimizing the cost
\begin{equation*}
J_t(\xi) =\left (
  \begin{array}[c]{l}
 \int_0^T \left(F[m(s)]   (\xi(s),\xi'(s))+  L(\xi(s),\xi'(s),s) +\frac 1 q   \left|\xi''(s) \right|^q \right) ds\\ \ds +  G[ m(T)](\xi(T), \xi(T)) 
  \end{array}
  \right),
\end{equation*}
over the trajectories  $\xi\in \Gamma_t[x,v]$, see \eqref{eq:3},
while  $m(s)\in  \cP(\Xi)$  is  the distribution of  states of the population of agents at time $s$, ($m(0)$ is prescribed). Note that given  $m\in C^0([0,T]; \cP(\Xi))$, the optimal control problem described above can be written as \eqref{eq:6} with $ \ell(x,v,s)= L(x,v,s)   +F[m(s)](x,v)$ and    $g(x,v)= G[m(T)](x,v)$. Thus, all the results in the previous sections can be applied if  $m\in C^0([0,T]; \cP(\Xi))$.

 Let us briefly explain why describing  the present MFG by a system of forward-backward  PDEs brings difficulties.
\begin{itemize}
\item Given $m$, with $\ell$ and $g$ defined immediately above, we expect that the value function of the latter optimal control problem is a viscosity solution of \eqref{eq:1000} in $ (0,T)\times \Omega \times \R^n$,  and satisfies  \eqref{eq:1001}. As said in the introduction, boundary conditions linked to state constraints would be  needed in order to characterize $u$,  but we are not aware of a complete theory on viscosity solutions in the present case.
\item Formally, the state distribution $m$ should be a weak solution of the following continuity equation:
   \begin{eqnarray*}
  \partial_t m +v\cdot D_xm-{\rm div}_v( m |D_v u|^{p-2}  D_vu)=0, \textrm{ in }
 (0,T)\times \Omega \times \R^n,\\
 m(x,v,0)=m_0(x,v),\textrm{ in }\Omega \times \R^n,
   \end{eqnarray*}
    with suitable boundary conditions on $(0,T)\times \partial \Xi$, 
   where $m_0\in \cP(\Xi)$ is the given initial distribution.
   As already said in the introduction, $m$ may develop singularities, especially on $\partial \Xi$, which makes the study of the forward-backward system difficult.
\end{itemize}
For the reasons given above,  we are rather going  to focus on  {\sl  relaxed equilibria} 
that we now describe.

Recall that  $\Gamma=\Gamma_0$ (defined in \eqref{eq:1} with $t=0$)  is a metric space with the distance $d(  \xi, \tilde\xi)= \|\xi-\tilde\xi\|_{  C^1([0,T];\R^n)}$. For $t\in [0,T]$, the evaluation map $e_t:\Gamma\to \Xi$ is defined by $e_t(\xi)=  (\xi(t), \xi'(t)) $ for all $  \xi  \in \Gamma$.

Let $\cP(\Gamma)$ be the set of Borel probability measures on $\Gamma$.
 For any $\mu\in \cP(\Gamma)$, let the Borel probability measure $m^\mu(t)$   on $ \Xi$ be defined by $m^\mu(t)=e_t\sharp \mu$. 
It is possible to prove that if $\mu\in \cP(\Gamma)$, then  $t\mapsto m^\mu(t)$ is continuous from $[0,T]$ to  $\cP(\Xi)$, for the narrow convergence  in $ \cP(\Xi)$.

For $m_0\in \cP(\Xi)$, let us define $ \cP_{m_0}(\Gamma)=\{\mu\in \cP(\Gamma): \;  e_0\sharp \mu =m_0\}$.  

With $\mu\in \cP(\Gamma)$, we associate the cost
\begin{equation}
\label{costMFG}
J^\mu( \xi  )
=\left (
  \begin{array}[c]{l}
\ds \int_0^T \left(F[m^\mu(s)]   (\xi(s),\xi'(s))+  L(\xi(s),\xi'(s),s) +\frac 1 q   \left|\xi''(s) \right|^q \right) ds\\ \ds +  G[ m^\mu(T)](\xi(T), \xi(T))    
  \end{array}
\right).
\end{equation}
Then,  $\Gamma^{\mu,\rm{opt}}[x,v]$ is the set of all $ \xi \in  \Gamma[x,v]$ such that $
\xi \in W^{2,q}(0,T, \R^n) $ and
$\xi$ achieves the minimum of $J^\mu$ in $\Gamma[x,v]$.

Then, the abstract definition of {\sl relaxed mean field equilibria} is identical  to that given in \cite{MR3888967} (recall that \cite{MR3888967} dealt with a different class of
state constrained MFGs with  locally controllable dynamics):
\begin{definition}
  \label{sec:setting-notation-3}
The probability measure $\mu \in \cP_{m_0}(\Gamma)$ is a relaxed constrained mean field equilibrium associated with the initial distribution $m_0$ if 
\begin{equation}\label{eq:49}
  {\rm{supp}}(\mu)\subset \mathop \bigcup_{
(x,v)\in   {\rm{supp}}(m_0)      }\Gamma^{\mu,{\rm opt}}[x,v].
\end{equation}
\end{definition}

\paragraph{\bf Existence of relaxed equilibria}
\begin{theorem}
  \label{sec:mean-field-games-5}
 Consider the following cases:
\begin{description}[style=unboxed,leftmargin=0cm]
\item {\bf Case 1:} $n=1$, $q>1$,  $\Omega=(-1,1)$
\item {\bf Case 2:} $n>1$, $q=2$,  $\Omega $ is a bounded and convex domain  with a $C^3$ boundary,
\end{description}
and $L$, $F$ and $G$ satisfying the assumptions in the paragraph above.
  
Let $m_0$ be a probability measure on $\Xi$ such that 
\begin{equation}
  \label{eq:50}
m_0(\Xi \setminus \Xi^{ \rm{ad}} )=0,
\end{equation}
where $ \Xi^{ \rm{ad}}$ is defined in   \eqref{eq:5}, (see also \eqref{eq:7} in Case 1).

Then there exists a relaxed constrained mean field  equilibrium  associated with the initial distribution $m_0$, i.e.  a probability measure $\mu \in \cP_{m_0}(\Gamma)$ such that  (\ref{eq:49}) holds.
\end{theorem}

\begin{proof}
  The proof mostly relies on
  \begin{itemize}
  \item   Propositions  \ref{sec:bounds_opt_traj_nd} and \ref{prop_closed_graph_1d} in
Case 1
  \item  Propositions  \ref{sec:bounds_opt_traj_nd} with $q=2$ and \ref{prop:clos-graph-prop} in  Case 2.
  \end{itemize}
It follows  all the steps in \cite[Section 4.3]{MR4444572}. Therefore, for brevity, we just sketch it and refer to  \cite[Section 4.3]{MR4444572} for the details.  It mainly consists of two steps:
  \begin{description}[style=unboxed,leftmargin=0cm]
  \item[1.]   The first step consists of proving the result when $m_0$ is supported in $\Theta_r$ for a given $r>0$, where $\Theta_r$ is defined in \eqref{eq:9}  (see also \eqref{eq:12} in Case 1), by means of Kakutani's fixed point theorem.  This step requires to use  both
  Propositions  \ref{sec:bounds_opt_traj_nd} and \ref{prop_closed_graph_1d} in
Case 1, and both  Propositions  \ref{sec:bounds_opt_traj_nd} and \ref{prop:clos-graph-prop} in Case 2.
    \item[2.] The second step in the proof consists  of approximating the measure $m_0$ by a sequence of probability measures $(m_{0,n})_{n>0}$, such that $ m_{0,n} $ is 
supported in $\Theta_n$.   Hence, from  the first step, we obtain a sequence of relaxed  equilibria $\mu_n$ 
related to $m_{0,n}$, which can be proved to be tight in $\cP(\Gamma)$. A relaxed mean field equilibrium $\mu$ related to $m_0$ is then  obtained as a cluster point of the sequence $\mu_n$  for the narrow convergence in $\cP(\Gamma)$. The assumption on $m_0$ implies that $\mu\in \cP_{m_0}(\Gamma)$. 
The fact that $ {\rm{supp}}(\mu) \subset \bigcup_{(x,v)\in {\rm{supp}}(m_0)}  \Gamma^{\mu, {\rm{opt}}}[x,v]$ is obtained by a key argument, see \cite[Prop. 4.10]{MR4444572},
which is a generalization of either Proposition \ref{prop_closed_graph_1d} in Case 1 or
Proposition  \ref{prop:clos-graph-prop} in Case 2,
to sequences  $\xi_i\in  \Gamma^{\mu_i, {\rm{opt}}}[x_i,v_i]$ such that 
$\mu_i $ narrowly converges to $\mu$ and $\xi_i$ converges to $\xi$ in $C^1$.
\end{description}
\end{proof}

\begin{remark}[In Case 1]
  Compared to \cite[Th. 2.8]{MR4444572}, we have made two improvements:
  \begin{enumerate}
  \item We have got rid of an  assumption made on $\ell$ in \cite{MR4444572}  which essentially meant  that the cost $\ell$ did not reward the trajectories that exit the domain
  \item The present theorem holds for any $q$, $1<q<\infty$, while \cite[Th. 2.8]{MR4444572} was proved only for $q=2$. 
  \end{enumerate}
\end{remark}

\begin{remark}[In Case 2]
  Compared to \cite[Th. 3.6]{MR4444572}, the present result is more general because
  we have got rid of the assumption that the support of $m_0$ is  contained in a compact set $\Theta$ that satisfies the assumption of Lemma \ref{sec:bounds_and_continuity}.

  The result is limited to $q=2$, because it relies on Proposition \ref{prop:clos-graph-prop} and point (ii) in Proposition \ref{lemma:ass_infimum}.
\end{remark}

\appendix

\section{Proofs of results contained in Subsection \ref{sec:auxil_pb}}
Before proving Proposition  \ref{sec:prop_1}, let us give a useful technical lemma:
 \begin{lemma}
\label{lem_2}
  Consider  $q\in (1,+\infty)$.
    \begin{enumerate}
    \item     The function $f$ defined in \eqref{eq:13} is continuous and increasing, smooth on $(1,+\infty)$ and 
is  a bijection from $[1,+\infty)$ onto $[-1,-1/2)$ with an increasing and continuous inverse.
\item The derivative of $f$ is defined on $(1,+\infty)$ by 
  \begin{equation}
    \label{eq:52}
    f'(y) =
    \frac { (q-1)^2 \left( y^{\frac q  {q-1} } -(y-1)^{\frac q  {q-1} } \right)^2 -q^2 y^{\frac 1 {q-1}} (y-1)^{\frac 1 {q-1}}} {(q-1)    \left( y^{\frac q  {q-1} } -(y-1)^{\frac q  {q-1} } \right)^2 }
  \end{equation}
 and its extension by  $q-1$ at $y=1$ is continuous on $[1,+\infty)$.
\end{enumerate}
\end{lemma}

\begin{proof}[Proof of Lemma \ref{lem_2}]
  It is clear that $f$ continuous in $[1,+\infty)$, smooth in
  $ (1,+\infty)$.  By straightforward calculus, we see that $\lim_{y\to +\infty} f(y)=-1/2$.

  Point 2. also stems from easy calculus.
  
  Let us now check that $f$ is increasing in $[1,+\infty)$. Assume that $y_\circ\in (1,+\infty)$ is a root of $f'$ and set $z= (1-1/y_\circ)^{\frac 1 {q-1}}\in [0,1)$. Standard calculus leads to 
        \begin{equation}
          \label{eq:53}
        (q-1)^2  (1-z^q)^2 -q^2 z (1-z^{q-1})^2=0.  
        \end{equation}
        By studying the variations of the function in \eqref{eq:53}  (one needs to compute its first and second derivatives), it can be seen that it takes positive values in $[0,1)$. Thus, $f$ has no critical points in $(1,+\infty)$. Combining this fact with the first observation implies that $f$  is increasing.   Therefore, $f$ is a bijection from $[1,+\infty)$ onto $[-1,-1/2)$ and has  an increasing and continuous inverse.
    \end{proof}

\begin{proof}[Proof of Proposition  \ref{sec:prop_1}]
       The set  $K_{\theta,w}$ is  convex and closed.  We will see below that it is non empty. Hence,  problem  \eqref{eq:15} is the minimization of a coercive, strictly convex and continuous functional over a non-empty, convex and closed set given by linear constraints.
    Thus,  there exists a unique minimizer denoted $\eta_{\theta, w}$.
In the remainder of the proof, we shall omit the indices $\theta$ and $w$ for brevity and simply set $\eta=\eta_{\theta,w}$.

    The Euler-Lagrange first order conditions read as follows: there exists a  real number $\mu$ such that $\eta $ is a weak solution of the variational inequality
   \begin{equation}
     \label{eq:54}
   \left\{  \begin{array}[c]{rcll}
      -\frac d {dt}\left(  \left|\eta'\right|^{q-2} \eta' \right) &\ge& -\mu, &\hbox{ in } (0, \theta) ,      \\
       \eta &\ge&  w ,  &\hbox{ in }  (0, \theta),  \\  
       ( -\frac d {dt}\left(  \left|\eta'\right|^{q-2} \eta' \right)+ \mu)( \eta- w) &=&0,   &\hbox{ in } (0, \theta) ,      \\
\ds         x+\int_0^\theta \eta(s) ds &\le& 0, \\\mu&\ge& 0,\\
       \ds     \mu\left( x+\int_0^\theta \eta(s) ds\right) &=&0,\\
       \eta(0)&=&v, \quad \hbox{and} \quad        \eta(\theta)=w.
     \end{array}     \right.
 \end{equation}
 \begin{description}[style=unboxed,leftmargin=0cm]
  \item[1.] Let us focus on the solutions such that  $\mu >0$ and  $x+\int_0^\theta \eta(s)ds=0$.

   Since $v>w=\eta(0)$ and  $\eta$ is continuous,  we may  define $\bar t\in (0,\theta]$, the minimal time  such that $\eta(\bar t)=w$. Note that $\eta(t)>w$ for all $t\in [0,\bar t)$.
     From the third and fifth lines in  \eqref{eq:54},  $   \frac d {dt}\left(  \left|\eta'\right|^{q-2} \eta' \right) = \mu\ge 0$ in $(0,\bar t)$. Hence, $  \left|\eta'\right|^{q-2} \eta'$ coincides with an affine and nondecreasing function a.e. in $(0,\bar t)$. Since $p \mapsto |p|^{q-2} p$ is a strictly increasing function defined on $\R$, this implies that $\eta'|_{(0,\bar t)}$ can be identified with a continuous and nondecreasing function defined in $[0,\bar t]$. We may thus introduce the parameter $\zeta= \eta'(0)$. We have obtained in particular that  $\eta'\ge \zeta$ in  $[0,\bar t]$.

     We claim that $\zeta$ is negative. Indeed, if $\zeta$ was nonnegative, then, from the previous observation, $\eta$ would be nondecreasing in $[0,\bar t]$, thus  $\eta(\bar t)\ge v>w$,  which would contradict the definition of $\bar t$.

     Since $\zeta<0$, $\eta'$ is negative in some interval containing $0$. With the same argument as above, we see that
 $\eta'$ cannot change sign in $[0,\bar t)$. Hence, $\eta'$ is negative in $  [0,\bar t)$. We deduce that
 $\eta' (t)=  -( (-\zeta)^{q-1} -\mu t)^{\frac 1 {q-1}}$ in $(0, \bar t)$.
  \begin{description}[style=unboxed,leftmargin=0cm]
 \item[a.]  Let us look for the solutions that satisfy the further conditions: $\bar t <\theta$ and  $\eta(s)=w$ for $s\ge \bar t$.
    From the arguments above, $\eta'(\bar t_-)$ is well defined.
If $\eta'(\bar t_-)<0$, then
$ \frac d {dt}\left(  (-\eta') ^{q-1} \right) $ would comprise a Dirac mass at $\bar t$ with a negative coefficient, which would contradict the inequality $ -\frac d {dt}\left(  \left|\eta'\right|^{q-2} \eta' \right) \ge -\mu$. Hence
 $ \eta'(\bar t)=0$. This implies that $\bar t = (-\zeta)^{q-1} / \mu$. We deduce that
 $\eta' (t)= \zeta  (1 - t /{\bar t} )_+^{\frac 1 {q-1}}$ and $\eta(t) = v +
 \frac {q-1}{q} \zeta \bar t \left( 1 -   (1 -  t /{\bar t} )_+^{\frac q {q-1}}\right)$ in $(0, \theta)$.
 Since $\eta(\bar t)=w$, we obtain that $  \bar t =   \frac {q}{q-1}   \frac {w-v}{\zeta} $ and that
 $\eta(t) = w +(v-w)   (1 -  t /{\bar t} )_+^{\frac q {q-1}}$.
 The identity  $x+\int_0^\theta \eta(s)ds=0$ implies
 that $ 0=x+\theta w + \bar t  \frac{(q-1)}{(2q-1)}  (v-w)  $, then
 $\bar t = - \frac {2q-1}{q-1} \frac {x+\theta w}{v-w} $. The inequality  $\bar t < \theta$ holds if and only if
 $ - \frac {2q-1}{q-1} \frac {x+\theta w}{v-w} < \theta$, i.e. $  \theta \in  \left(  \frac {(2q-1)|x|}{(q-1)v+qw}  ,T\right]$.

 We finally obtain that  $\zeta= \frac q {2q-1} \frac { (v-w)^2}{x+\theta w}$ and $\mu =\frac {q-1} q  \frac 1 {v-w} \left( -\frac q {2q-1} \frac { (v-w)^2}{x+\theta w} \right)^q$. 
 This leads to the third  line of \eqref{eq:16}, namely
$
I(\theta,w)=\frac {q^{q-1}} {(2q-1)^q}   \frac {(w-v)^{2q-1}} { (|x|-\theta w)^{q-1}}$.
 We have found the solution of  \eqref{eq:54} and the minimizer of \eqref{eq:15} in the case when  $\theta \in  \left( \frac {(2q-1)|x|}{(q-1)v+qw}  ,T\right]  $.
 
\item[b.]  Let us look for the solutions that satisfy the further condition  $\bar t =\theta$.  We know that for all $t\in [0,\theta]$,
  $
  \eta' (t)= -( (-\zeta)^{q-1} -\mu t)^{\frac 1 {q-1}}$.
  Taking the primitive, we obtain that  for all $t\in [0,\theta]$,
  \begin{equation}
    \label{eq:55}
  \eta(t)= v   -\frac {q-1} q \frac {(-\zeta)^{q}}{\mu} \left(1- \left(1-   \frac \mu {(-\zeta)^{q-1}} t\right)    ^{\frac q {q-1}}\right).
  \end{equation}
The identity  $\eta(\theta)=w$ and \eqref{eq:55} yield
    \begin{equation}\label{eq:56}
      1-\frac \mu {(-\zeta)^{q-1}} \theta = \left(  1+  \frac q {q-1}  (w-v)   \frac{\mu} {(-\zeta)^{q}}    \right)^{\frac{q-1} q}.
    \end{equation}
Next, from the identity  $x+\int_0 ^\theta \eta(t)dt=0$ and \eqref{eq:55}, we see that 
\begin{eqnarray*}
        0=    x+v\theta -\frac {q-1} q \frac {(-\zeta)^{q}}{\mu} \theta
             +\frac {(q-1)^2}{q(2q-1)}\frac {(-\zeta)^{q}}{\mu}
      \frac {(-\zeta)^{q-1}}{\mu}  \left(  1  - \left(   1-\frac \mu {(-\zeta)^{q-1}} \theta     \right) ^{\frac {2q-1}{q-1}}\right)  .
\end{eqnarray*}
      Using \eqref{eq:56}, this becomes
      \begin{equation} \label{eq:57}   
        \begin{array}[c]{rcl}
            0&=&
        x+v\theta -\frac {q-1} q \frac {(-\zeta)^{q}}{\mu} \theta
        +\frac {(q-1)^2 } {q(2q-1)}   \frac {(-\zeta)^{q}}{\mu}  \frac {(-\zeta)^{q-1}}{\mu}\\
             && -\frac {(q-1)^2 } {q(2q-1)}   \left(   \frac {(-\zeta)^{q-1}}{\mu} -\theta  \right) \left( \frac {(-\zeta)^{q}}{\mu} +\frac q {q-1} (w-v) \right) 
          \\
        &=&  x+\theta\left(\frac  q {2q-1} v +\frac {q-1}{2q-1} w \right)
          -\frac {q-1}{2q-1}   \frac {(-\zeta)^{q}}{\mu} \theta
           +\frac {q-1}{2q-1}(v-w)   \frac {(-\zeta)^{q-1}}{\mu}.
         \end{array}
       \end{equation}
     After changing the unknowns from $(\zeta, \mu)$ to $(X, Y)$, where
      \begin{equation}
        \label{eq:58}
        X=  \frac {(-\zeta)^{q-1}}{\mu \theta}\quad  \hbox{and} \quad Y=  \frac {(-\zeta)^{q}}{\mu},
      \end{equation}
 we deduce from  \eqref{eq:57}  and  \eqref{eq:56}  the following system
 \begin{eqnarray}
   \label{eq:59}
   \frac {q-1}{2q-1}(v-w) X -\frac {q-1} {2q-1} Y
   &=&  -\frac x \theta -   \frac {q-1}{2q-1} w - \frac q {2q-1} v,\\
   \label{eq:60}
   Y&=&\frac {q}{q-1} (v-w) \frac 1 { 1-(1-1/X)^{\frac q {q-1}}},
 \end{eqnarray}
 and after eliminating $Y$, we find that $X$ must be a solution of \eqref{eq:17}.

 Observe that $-\frac 1{v-w} \left((2q-1)\frac x \theta +(q-1)w +qv\right) $ belongs to $ [-1,-1/2)$ if and only if $ \theta\in \left   ( \frac {2|x|}{v+w},  \frac {(2q-1)|x|}{(q-1)v+qw} \right]$.  We deduce from Lemma \ref{lem_2} that \eqref{eq:17}  has a unique solution (that  depends continuously on the right-hand side) if and only if
 $ \theta\in \left   ( \frac {2|x|}{v+w},  \frac {(2q-1)|x|}{(q-1)v+qw} \right]$.
 Under the latter condition,  \eqref{eq:58} and \eqref{eq:60} imply 
   \begin{equation}
     \label{eq:61}
     \mu^{\frac q {q-1}}  =\frac {Y^{q}}{X^{\frac {q^2} {q-1}} \theta^{\frac {q^2} {q-1}}}
     = \left(\frac {q}{q-1}\right)^q \frac {(v-w)^q} { \theta^{\frac {q^2} {q-1}}}  \left(X^{\frac q {q-1}}-(X-1)^{\frac q {q-1}}\right)^{-q}.
   \end{equation}
  Then,  $I(\theta, w)$  can be expressed in terms of $X$  (the solution of  \eqref{eq:17})  and $\mu$ (given by  \eqref{eq:61}) as follows:
   \begin{equation*}
     \begin{split}
       I(\theta, w)= \frac 1 q \int_0 ^\theta \left( (-\zeta)^{q-1} -\mu s\right)^{\frac q {q-1}} ds
       &= \frac 1 q \frac {q-1}{2q-1} \left(  \frac {  (-\zeta)^{2q-1}} \mu  -      \frac { \left( (-\zeta)^{q-1}-\mu \theta \right)^{\frac {2q-1}{q-1 }}} \mu \right) \\
       &= \frac 1 q \frac {q-1}{2q-1} \mu ^{\frac q {q-1}} \theta^{\frac {2q-1}{q-1}}
 \left( X^{\frac {2q-1}{q-1}} -(X-1)^{\frac {2q-1}{q-1}}\right)  .   \end{split}
\end{equation*}
Plugging  \eqref{eq:61} yields 
\begin{equation}
  \label{eq:62}
     I(\theta, w)=\frac 1 {2q-1} \left( \frac {q}{q-1}\right)^q (v-w)^q \theta^{-(q-1)} \frac { X^{\frac {2q-1}{q-1}} -(X-1)^{\frac {2q-1}{q-1}} }{ \left( X^{\frac {q}{q-1}} -(X-1)^{\frac {q}{q-1}}\right)^q}.
   \end{equation}
   We have found the solution of  \eqref{eq:54} and the minimizer of \eqref{eq:15} in the case when  $\theta \in
\left( \frac {2|x|}{v+w},  \frac {(2q-1)|x|}{(q-1)v+qw} \right] $, and have obtained the second line of \eqref{eq:16}.

In the special case $q=2$, it was proved in \cite[Lemma 4.4]{MR4444572}  that
$   \eta_{\theta, w} (t)= \ds v + \zeta t  +\frac \mu 2  t ^2 $    with
$ \zeta= -\frac{6x+(4v+2w) \theta}{\theta^2}$ and $\mu =6 \frac {2x+(v+w)\theta} {\theta^3}$,
which leads to \eqref{eq:18}.
    \end{description}
  \item[2.] If  $\frac {|x|} \theta \ge \frac {v+w} 2$, the solution of \eqref{eq:54} is given by $\mu=0$ and $ \eta(t)=  v -  (v-w) \frac t \theta$.
  Then, $ I(\theta,w)=\frac 1 q  \frac {(w-v)^{q}} { \theta^{q-1}}$. We have obtained the first line of \eqref{eq:16}.
   \end{description}

   \medskip
   
   Finally, we deduce from the explicit form of  $I(\theta, w)$ when it is available and from the observation  on the continuity of $X$, that $I$ is continuous on the domain $(0,T]\times [0,v)$. \end{proof}

\begin{proof}[Proof of Lemma \ref{lem_var_w}]
  Recall that $\theta$ is fixed in   $\left (\frac {|x|} v, \frac {(2q-1) |x|}{ (q-1) v}
\right] $.
It is easy to check that $v> \frac { (2q-1)|x|/\theta-(q-1)v } q \ge  \max\left(0, 2\frac{|x|} \theta -v\right)$
and that  $0< \frac { (2q-1) (v -|x|/\theta) }{v} \le q$.
  
  First, we easily deduce from the explicit formulas in the  first and third lines  in \eqref{eq:16} that  $w\mapsto I(\theta, w)$ is decreasing in $\left[0,  \max \left (0, 2\frac{|x|} \theta -v\right)\right]$ and increasing in $\left[ \frac { (2q-1)|x|/\theta-(q-1)v } q ,v\right)$.

  Setting $ J:=  \left(  \max\left( 0, 2\frac{|x|} \theta -v\right),  \frac { (2q-1)|x|/\theta-(q-1)v } q\right)$,
    the main difficulty  occurs when $J\not = \emptyset$. If 
 $ w \in J$, then  $I(\theta, w)$ is given by the second line in \eqref{eq:16}, that reads 
     \[ I(\theta,w)=   \frac { \theta^{-(q-1)}} {2q-1} \left( \frac {q}{q-1}\right)^q   (v-w)^q g(X(w)),\]
where
\begin{equation}
  \label{eq:63}
  g(X)= \frac { X^{\frac {2q-1}{q-1}} -(X-1)^{\frac {2q-1}{q-1}} }{ \left( X^{\frac {q}{q-1}} -(X-1)^{\frac {q}{q-1}}\right)^q}
\end{equation}
and $X(w)$ is the unique solution to \eqref{eq:17}. Also, \eqref{eq:17} reads
$f(X(w))= h(w)$, with
\begin{displaymath}
  h(w) =   -\frac 1{v-w} \left((2q-1)\frac x \theta +(q-1)w +qv\right).
\end{displaymath}
Thus, the desired result will  be obtained if we can prove that $J\ni w\mapsto   (v-w)^q g( X(w)) $ is decreasing, or after differentiating (which is permitted), that
\begin{equation}
  \label{eq:64}
   (v-w)   h'(w)    \frac {g' \left( X(w)\right)} {  f'( X(w)) \; g \left( X(w)\right) }        < q,\quad  \forall w\in J.
\end{equation}
After some calculus left to the reader, one gets that for any $X>1$,
\begin{displaymath}
  \frac {g' (X)} {g (X)}=  
\frac { (q-1)^2 \left( X^{\frac q  {q-1} } -(X-1)^{\frac q  {q-1} } \right)^2
  -q^2 X^{\frac 1 {q-1}} (X-1)^{\frac 1 {q-1}}}
{(q-1)    \left( X^{\frac q  {q-1} } -(X-1)^{\frac q  {q-1} } \right)
  \left( X^{\frac {2q-1}  {q-1} } -(X-1)^{\frac {2q-1}  {q-1} } \right)      },
\end{displaymath}
which, combined with \eqref{eq:52}, gives
\begin{displaymath}
  \frac {g' (X)} {f'(X) g (X)}=   \frac {X^{\frac q  {q-1} } -(X-1)^{\frac q  {q-1}}}{ X^{\frac {2q-1}  {q-1} } -(X-1)^{\frac {2q-1}  {q-1}}}.
\end{displaymath}
After studying the right-hand side in the latter identity for $X\in (1,\infty)$, we deduce that  $ {g' (X)} /  ({f'(X) g (X)})$ takes its values in $(0,1)$ and tends to $1$ as $X\to 1_+$.
On the other hand,  $  (v-w)   h'(w)= \frac { (2q-1) (v -|x|/\theta) }{v-w} \in (0,q)$ for $w\in J$. We  deduce \eqref{eq:64} from the latter two points. Thus $w\mapsto I(\theta,w)$ is decreasing in $J$.

The minimum of $w\mapsto I(\theta,w)$ is achieved at $w_{\min}= \frac { (2q-1)|x|/\theta-(q-1)v } q$. Then  $X(w_{\min})=1$ and  $I(\theta,w_{\min})= \frac { (2q-1)^{q-1}}{q (q-1)^{q-1}} \frac {(v -|x|/\theta)^q}{\theta^{q-1}}$. 
\end{proof}

\begin{proof}[Proof of Proposition \ref{lemma:ass_infimum} ]
  Let us make the change of variables $\tau=\theta v_i/|x_i|$, $z= w_i/v_i$, and define the new function $J_i: \left(0, \frac  {v_i T}{|x_i|}\right] \times [0, 1)\to \R$, $$J_i(\tau, z)= \frac { |x_i|^{q-1}}  {v_i^{2q-1}} I_i\left(    \frac { |x_i|}{v_i}\tau, z v_i\right).$$
We deduce from \eqref{eq:16} that 
  \begin{equation}\label{eq:65}
       J_i(\tau,z)= \left\{  \begin{array}[c]{rl}
         \ds \frac 1 q  \frac {(1-z)^{q}} { \tau^{q-1}} ,
         \quad \quad\quad &\hbox{if}\quad  \tau\in \left
                                                                                (0,  \frac {2}{1+z}\right],
         \\ 
                 \ds \frac {q^{q-1}} {(2q-1)^q}   \frac {(1-z)^{2q-1}} { (1-\tau z)^{q-1}}
 , \quad\quad \quad &\hbox{if}\quad  \tau \in  \left[  \frac {(2q-1)}{(q-1)+qz}  ,\frac {v_iT} {|x_i|} \right).
  \end{array}\right.
\end{equation}
We are only interested in small values of $z$ because, from the assumptions, $|w_i|\ll v_i$, so we can focus on $z\in [0,1/2]$. On the other hand, 
  we know that the minimum of $\theta\mapsto I_i (\theta, w_i)$ is reached in the interval
 $\left[ \frac {2|x_i|}{v_i+w_i},   \frac {(2q-1)|x_i|}{(q-1)v_i+qw_i} \right]$
  which is contained in $
  \left[ \frac {|x_i|}{v_i},   \frac {(2q-1)|x_i|}{(q-1)v_i}\right]$, so we can limit ourselves to studying the restriction of $J_i$ to  $ \left[1,  \frac {2q-1}{q-1}\right]\times [0,1/2]$.

  Recall from the proof of Proposition \ref{sec:prop_1} that
  the map from $[-1,-1/2)$ to $[1,+\infty)$ which maps  the right-hand side of  \eqref{eq:17} to its unique solution $X$  is continuous, and that the right-hand side of  \eqref{eq:17} takes the form  $ -\frac 1{v-w} \left((2q-1)\frac x \theta +(q-1)w +qv\right)$.  Moreover, if $\tau\in (\frac 2 {1+z},\frac {2q-1}{q-1+qz}]$, then   \[ -\frac 1{v_i-w_i} \left((2q-1)\frac {x_i} \theta +(q-1)w +qv_i\right) = \frac {  \frac {2q-1} \tau - (q-1) z -q} {1-z}\in [-1,-1/2).\] We deduce from the latter three points and from \eqref{eq:65} that  $J_i$ is continuous in the  rectangular domain $ \left[1,  \frac {2q-1}{q-1}\right]\times [0,1/2]$,  uniformly  with respect to $i$.

Hence,
$
\lim_{z\to 0} \inf_{\tau } J_i(\tau, z)= \inf_{\tau } J_i(\tau, 0)= \frac {q^{q-1}} {(2q-1)^q}$,
where the latter identity comes from Lemma \ref{lemma:ass_infimum0}.

Finally, because $ \frac  {v_i^{2q-1}}  { |x_i|^{q-1}}$ is bounded from below by a positive constant $c$, we deduce
\eqref{eq:19} for $0\le w_i\ll v_i$. We have proved (i).
In the quadratic case, i.e. $q=2$, the  more accurate result  \eqref{eq:20} can be found by using the explicit formula  \eqref{eq:18}, see \cite[Lemma 4.5]{MR4444572} and its proof.
\end{proof}

 \section*{Acknowledgments}
  The author was  partially supported by the ANR (Agence Nationale de la Recherche) through project ANR-16-CE40-0015-01 and  by the chair Finance and Sustainable Development and FiME Lab (Institut Europlace de Finance).   The author would like to thank Nicoletta Tchou for helpful discussions.

\bibliographystyle{siam}
\bibliography{State_constrained_opt_contr_acc}

\end{document}